\documentclass[a4paper,reqno]{amsart}
 \usepackage[utf8]{inputenc}
 \usepackage{a4wide}
 \usepackage{amssymb,exscale,bbm,mathrsfs,stmaryrd,xargs, units}
 \usepackage[mathscr]{euscript}
 \usepackage{color}
 \usepackage[symbol]{footmisc} 

 \newtheorem{theorem}{Theorem}[section]
 \newtheorem{corollary}[theorem]{Corollary}
 \newtheorem{lemma}[theorem]{Lemma}
 \newtheorem{proposition}[theorem]{Proposition}
 \theoremstyle{definition}
 \newtheorem{definition}[theorem]{Definition}
 \theoremstyle{remark}
 
 \newtheorem{example}[theorem]{Example}
 
 \numberwithin{equation}{section}

\begin{document}

%
\setlength{\parindent}{0cm}
\renewcommand{\labelenumi} {(\alph{enumi})}    
\renewcommand{\labelenumii}{(\roman{enumii})}
\renewcommand{\theenumi} {(\alph{enumi})}      
\renewcommand{\theenumii}{(\roman{enumii})}    
  \renewcommand {\Re}{\text{Re}}
  \renewcommand {\Im}{\text{Im}}
  \newcommand   {\upi}{{\mathrm \pi}}
  \newcommand   {\ud} {{\mathrm d}}
  \newcommand   {\dr}{{\mathrm d}r}
  \newcommand   {\ds}{{\mathrm d}s}
  \newcommand   {\dt}{{\mathrm d}t}
  \newcommand   {\dx}{{\mathrm d}x}
  \newcommand   {\dxi}{{\mathrm d}\xi}
  \newcommand   {\dy}{{\mathrm d}y}
  \newcommand   {\dz}{{\mathrm d}z}
  \newcommand*{\longhookrightarrow}{\ensuremath{\lhook\joinrel\relbar\joinrel\relbar\joinrel\rightarrow}}

  \newcommand {\pihalbe}{\nicefrac{\upi}{2}}
  \newcommand {\einhalb}{{\nicefrac{1}{2}}}
  \newcommand {\CC}{\mathbb C}  
  \newcommand {\RR}{\mathbb R}
  \newcommand {\KK}{\mathbb K}  
  \newcommand {\NN}{\mathbb N}
  \newcommand {\ZZ}{\mathbb Z}

  \newcommand {\BOUNDED}{\mathcal B}
  \newcommand {\DOMAIN}{\mathscr D}
  \newcommand {\eins} {\mathbbm 1}
  \newcommand {\al}{\alpha}
  \newcommand {\la}{\lambda}
  \newcommand {\eps}{\varepsilon}
  \newcommand {\Ga}{\Gamma}
  \newcommand {\ga}{\gamma}
  \newcommand {\om}{\omega}
  \newcommand {\Om}{\Omega}
  \newcommand {\norm}[1] {\| #1 \|}  
  \newcommand {\abs}[1] {|#1|}
  \newcommand {\biggabs}[1] {\bigg|#1\bigg|}
  \newcommand {\lrnorm}[1]{\left\| #1 \right\|}
  \newcommand {\bignorm}[1]{\bigl\| #1 \bigr\|}
  \newcommand {\Bignorm}[1]{\Bigl\| #1 \Bigr\|}
  \newcommand {\Biggnorm}[1]{\Biggl\| #1 \Biggr\|}
  \newcommand {\biggnorm}[1]{\biggl\| #1 \biggr\|}
  \newcommand {\Bigidual}[3] {\Bigl\langle #1, #2 \Bigr\rangle_{#3}}
  \newcommand {\bigidual}[3] {\bigl\langle #1, #2 \bigr\rangle_{#3}}
  \newcommand {\idual}[3] {\langle #1, #2 \rangle_{#3}}
  \newcommand {\Bigdual}[2] {\Bigidual{#1}{#2}{}}
  \newcommand {\bigdual}[2] {\bigidual{#1}{#2}{}}
  \newcommand {\dual}[2] {\idual{#1}{#2}{} }
  \newcommand {\sector}[1] {S_{#1}}

  \newcommand {\SUCHTHAT}{:\;}
  \newcommand {\embeds} {\hookrightarrow}

  \newcommand {\ui}{\text{\rm i}}  
  \newcommand {\SCHWARTZ} {\mathscr S}
  \newcommand {\FOURIER} {\mathscr F}
  \newcommand {\dist}{\text{dist}}
  \newcommand {\SECTOR}[1]{S_{#1}} 
  \newcommand {\calA}{{\mathcal A}}
  \newcommand {\calL}{{\mathcal L}}
  \newcommand {\calQ}{{\mathcal Q}}
  \newcommand {\calS}{{\mathcal S}}
  \def\aDeForme{\mathop{\mbox{{\LARGE $\mathfrak{a}$}}}} 
  \newcommandx*\form[3][2=\cdot,3=\cdot]{ \aDeForme(#1; #2, #3)}
  \newcommandx*\sg[3][2=A,3=t]{e^{-{#1}\cdot {#2}(#3)}}
  \newcommand {\sprod}[2] {\left[#1\,|\,#2\right]_H}
  \newcommand {\dprod}[2] {\left\langle #1, #2 \right\rangle}
  \newcommand {\bracket}[1]{\langle{#1}\rangle}
   \newcommand {\supp}{\text{supp}}

\textheight=1.02\textheight  

\allowdisplaybreaks

\title[Maximal regularity for non-autonomous evolution equations ]{Maximal regularity for non-autonomous evolution equations }
\author[Bernhard H. Haak]{ Bernhard H. Haak}
\address{%
Institut de Math\'ematiques de Bordeaux, CNRS UMR 5251 \\Univ.  Bordeaux \\351, cours de
la Lib\'eration\\33405 Talence CEDEX\\FRANCE}
\email{bernhard.haak@math.u-bordeaux.fr, elmaati.ouhabaz@math.u-bordeaux.fr}
\author[El Maati Ouhabaz]{El Maati Ouhabaz}
\thanks{The research of both authors was partially supported by the ANR
project HAB, ANR-12-BS01-0013-02}

\begin{abstract}
 We consider the maximal regularity problem for non-autonomous evolution equations
 \begin{equation}
 \left\{
  \begin{array}{rcl}
     u'(t) + A(t)\,u(t) &=& f(t), \ t \in (0, \tau] \\
     u(0)&=&u_0.
  \end{array}
\right.
\end{equation}
Each operator $A(t)$ is associated with a sesquilinear form $\form{t}$
on a Hilbert space $H$.  We assume that these forms all have the same
domain and satisfy some regularity assumption with respect to $t$
(e.g., piecewise $\alpha$-H\"older continuous for some $\alpha >
\einhalb$).  We prove maximal $L_p$--regularity for all $u_0 $ in the
real-interpolation space $(H, \DOMAIN(A(0)))_{1-\nicefrac{1}{p},p}$. The
particular case where $p = 2$ improves previously known results and 
gives a positive answer to a question  of
J.L. Lions \cite{Lions:book-PDE} on the set of allowed initial data $u_0$. 
\end{abstract}

\maketitle

\section{Introduction and main results}
Let $H$ be a real or complex Hilbert space and let $V$ be another
Hilbert space with dense and continuous embedding $V \embeds H$.  We
denote by $V'$ the (anti-)dual of $V$ and by $\sprod{\cdot}{\cdot}$
the scalar product of $H$ and $\dprod{\cdot}{\cdot}$ the duality
pairing $V'\times V$.  The latter  satisfies (as usual)
$\dprod{v}{h} = \sprod{v}{h}$ whenever $v \in H$ and $h \in V$. By the
standard identification of $H$ with $H'$ we then obtain continuous and
dense embeddings $V \embeds H \eqsim H' \embeds V'$. We denote by
$\norm{.}_V$ and $\norm{.}_H$ the norms of $V$ and $H$, respectively.

\smallskip

We are concerned with the non-autonomous evolution equation
\begin{equation}\label{eq:evol-eq} \tag{P}
\left\{
  \begin{array}{rcl}
     u'(t) + A(t)\,u(t) &=& f(t), \ t \in (0, \tau] \\
     u(0)&=&u_0,
  \end{array}
\right.
\end{equation}
where each operator $A(t)$, $t \in [0, \tau]$,  is associated with a
sesquilinear form $ \form{t}$. 
Throughout this article we will assume that 
     \let\ALTLABELENUMI\labelenumi \let\ALTTHEENUMI\theenumi
     \renewcommand{\labelenumi}{[H\arabic{enumi}{]}}
     \renewcommand{\theenumi}{[H\arabic{enumi}{]}}
\begin{enumerate}
\item \label{item:constant-form-domain}
  (constant form domain) $\DOMAIN(\form{t}) = V$.
\item \label{item:uniform-continuity} (uniform boundedness) there
  exists $M>0$ such that for all $t \in [0, \tau]$ and $u, v \in V$,
  we have $|\form{t}[u][v]|\le M \norm{u}_V \norm{v}_V$.
\item \label{item:uniform-accretivity} (uniform quasi-coercivity)
  there exist $\al>0$, $\delta \in \RR$ such that for all $t \in [0,
  \tau]$ and all $u, v \in V$ we have $\al \norm{u}_V^2 \le \Re
  \form{t}[u][u] + \delta \norm{u}_H^2$.
\end{enumerate}
     \let\labelenumi\ALTLABELENUMI
     \let\theenumi\ALTTHEENUMI

     \smallskip Recall that $u \in H$ is in the domain $\DOMAIN(A(t))$
     if there exists $h\in H$ such that for all $v \in V$:
     $\form{t}[u][v] = \sprod{h}{v}$.  We then set $A(t)u := h$.  We
     mention that equality of the form domains, i.e.,
     $\DOMAIN(\form{t}) = V$ for $t \in [0, \tau]$ does not imply
     equality of the domains $\DOMAIN(A(t))$ of the corresponding
     operators. For each fixed $u \in V$, $\phi := \form{t}[u]$
     defines a continuous (anti-)linear functional on $V$, i.e.  $\phi
     \in V'$, then it induces a linear operator $\calA(t): V \to V'$
     such that $\form{t}[u][v] = \dprod{\calA(t)u}{v} $ for all $u, v
     \in V$. Observe that for $u \in V$,
\[
   \norm{ \calA(t)u }_{V'} 
    = \sup_{v \in V, \norm{v}_V=1} |\dprod{\calA(t)u}{v}| 
    =  \sup_{v \in V, \norm{v}_V=1} |\form{t}[u][v]| \le M \norm{u}_V
\]
so that $\calA(t) \in \BOUNDED(V, V')$. The operator $\calA(t)$ can be
seen as an unbounded operator on $V'$ with domain $V$ for all $t \in
[0, \tau]$.  The operator $A(t)$ is then the part of $\calA(t)$ on
$H$, that is, 
\[
   \DOMAIN(A(t)) = \{ u \in V,\  \calA(t) u \in H \}, \qquad A(t) u = \calA(t) u.
\]
It is a known fact that $-A(t)$ and $-\calA(t)$ both generate
holomorphic semigroups $(e^{-s \, A(t)})_{s \ge 0}$ and $(e^{-s\, \calA(t)})_{s \ge 0}$ on $H$ and $V'$, respectively. For each
$s \ge 0$, $e^{-s \, A(t)}$ is the restriction of $e^{-s \, \calA(t)}$ to $H$.  For all this, we refer to Ouhabaz
\cite[Chapter~1]{Ouhabaz:book}.

\medskip

The notion of maximal $L_p$--regularity for the above Cauchy problem
is defined as follows:
\begin{definition}\label{def:max-reg}
  Fix $u_0$.  We say that (\ref{eq:evol-eq}) has maximal
  $L_p$--regularity (in $H$) if for each $f \in L_p(0,\tau; H)$ there
  exists a unique $u \in W^{1}_p(0, \tau; H)$, such that $u(t) \in
  \DOMAIN(A(t))$ for almost all $t$, which satisfies
  (\ref{eq:evol-eq}) in the $L_p$--sense. Here $W^{1}_p(0, \tau; H)$
  denotes the classical $L_p$--Sobolev space of order one of functions
  defined on $(0, \tau)$ with values in $H$.
 \end{definition}

 \noindent Maximal regularity of an evolution equation on a Banach
 space $E$ depends on the operators involved in the equation, the
 space $E$ and the initial data $u_0$. The initial data has to be in
 an appropriate space. In the autonomous case, i.e., $A(t) = A$ for
 all $t \in [0, \tau]$, maximal $L_p$--regularity is well understood
 and it is also known that $u_0$ has to be in the real-interpolation
 space $(E, \DOMAIN(A))_{1-\nicefrac1{p}, p}$, see
 \cite{CannarsaVespri}.  We refer the reader further to the survey of
 Denk, Hieber and Pr\"uss \cite{DHP} and the references given therein.
 Note also that maximal regularity turns to be an important tool to
 study quasi-linear equations, see e.g.  the monograph of Amann
 \cite{Amann:1}.

 \noindent For the non-autonomous case we consider here, we first
 recall that if the evolution equation is considered in $V'$, then
 Lions proved maximal $L_2$--regularity for all initial data $u_0 \in
 H$, see e.g. \cite{Lions:book-PDE}, \cite[page 112]{Showalter}. This
 powerful result means that for every $u_0 \in H$ and $f \in L_2(0,
 \tau; V')$, the equation
\begin{equation}  \label{eq:lions-pb} \tag{P'}
  \left\{
  \begin{array}{rcl}
     u'(t) + \calA(t)\,u(t) &=& f(t) \\
     u(0)&=&u_0
  \end{array}\right.
\end{equation}
has a unique solution $u \in W^{1}_2(0,\tau ; V') \cap L_2(0, \tau
;V)$.  Note that this implies the continuity of $u(\cdot)$ as an
$H$--valued function, see \cite[XVIII Chapter 3, p.\ 513]{DL92}.  It
is a remarkable fact that Lions's theorem does not require any
regularity assumption (with respect to $t$) on the sesquilinear forms
apart from measurability. The apparently additional information $u \in
L_2(0, \tau ; V)$ follows from maximal regularity and the equation as
follows: for almost all $t$, $u(t) \in V$. For these $t \in (0, \tau)$
\begin{align*}
 \Re \form{t}[u(t)][u(t)]   & = \; -  \Re\dprod{u'(t)}{u(t)} + \Re \dprod{f(t)}{u(t)}\\
   &  \le \; \norm{u(t)}_V \norm{u'(t)}_{V'} +  \norm{u(t)}_V  \norm{f(t)}_{V'}.
\end{align*}
Suppose now that  the forms are coercive (i.e., $\delta = 0$ in \ref{item:uniform-accretivity}). Then  for some constant $c
> 0$ independent of $t$,
\begin{equation}  \label{eq:double-etoile}
      \norm{u(t)}_V^2 \le c\left[  \norm{u'(t)}_{V'}^2 +  \norm{f(t)}_{V'}^2 \right].
\end{equation}
Therefore, $u \in L_2(0, \tau; V)$ whenever $u \in W^{1}_2(0,\tau;
V')$ and $f \in L_2(0,\tau; V')$. If the forms are merely
quasi-coercive, we note that if $u(t)$ is the solution of
(\ref{eq:lions-pb}) then $u(t)e^{-\delta t}$ is the solution of the
same problem with $\calA (t) + \delta$ instead of $\calA(t)$ and
$f(t)e^{-\delta t}$ instead of $f(t)$. We apply now the previous
estimate \eqref{eq:double-etoile} to $u(t)e^{-\delta t}$ and $f(t)
e^{-\delta t}$ and obtain
\[
    \norm{u(t)}_V^2 \le c\left[  \norm{u'(t)}_{V'}^2  + \norm{u(t)}_{V'}^2 +  \norm{f(t)}_{V'}^2 \right]
\]
for some constant $c'$ independent of $t$.
\medskip

\noindent Note however that maximal regularity in $V'$ is not
satisfactory in applications to elliptic boundary value problems. For
example, in order to identify the boundary condition one has to
consider the evolution equation in $H$ rather than in $V'$.  For
symmetric forms (equivalently, self-adjoint operators $A(t)$) and $u_0
= 0$, Lions \cite[p. 65]{Lions:book-PDE}, proved maximal
$L_2$--regularity in $H$ under the additional assumption that $t
\mapsto \form{t}[u][v]$ is $C^1$ on $[0,\tau]$.  For $u(0) = u_0 \in
\DOMAIN(A(0))$ Lions \cite[p. 95]{Lions:book-PDE} proved maximal
$L_2$--regularity in $H$ for \eqref{eq:evol-eq} provided $t \mapsto
\form{t}[u][v]$ is $C^2$.  If the forms are symmetric and $C^1$ with
respect to $t$, Lions proved maximal $L_2$-regularity for $u(0) = u_0
\in V$ (one has to combine \cite[Theorem~1.1, p.~129 and Theorem~5.1,
p.~138]{Lions:book-PDE} to see this).  He asked the following
problems.

\medskip

\noindent{\bf Problem 1}: Does the maximal $L_2$--regularity in $H$
hold for \eqref{eq:evol-eq} with $u_0 = 0$ when $t \mapsto
\form{t}[u][v]$ is continuous or even merely measurable ?

\medskip
\noindent{\bf Problem 2}: For $u(0) = u_0 \in \DOMAIN(A(0))$, does the 
maximal $L_2$--regularity in $H$ hold under the weaker assumption that $t
\mapsto \form{t}[u][v]$ is $C^1$ rather than $C^2$ ?

\medskip 

\noindent The problem 1 is still open although some progress has
been made.  We mention here Ouhabaz and Spina \cite{OuhabazSpina} who
prove maximal $L_p$--regularity for \eqref{eq:evol-eq} when $u(0) = 0$
and $t \mapsto \form{t}[u][v]$ is $\alpha$-H\"older continuous for
some $\alpha > \frac{1}{2}$. More recently, Arendt et
al. \cite{ArendtDierLaasriOuhabaz} prove maximal $L_2$--regularity in
$H$ for
\[
\left\{
  \begin{array}{rcl}
     B(t)u'(t) + A(t)\,u(t) &=& f(t), \ t \in (0, \tau] \\
     u(0)&=&0 
  \end{array}
\right.
\]
in the case where $t \mapsto \form{t}[u][v]$ is piecewise Lipschitz
and $B(t)$ are bounded measurable operators satisfying $\gamma
\norm{v}_H^2 \le \Re \sprod{B(t)v}{v} \le \gamma' \norm{v}_H^2$ for some
positive constants $\gamma$ and $\gamma'$ and all $v \in H$. The
multiplicative perturbation by $B(t)$ was motivated there by
applications to some quasi-linear evolution equations.

\noindent 
Concerning the problem with $u_0 \neq 0$ and forms which are not necessarily
symmetric, Bardos \cite{Bar71} gave a partial positive answer to Problem
2 in the sense that one can take the initial data $u_0$ in $V$ under
the assumptions that the domains of both $A(t)^{\einhalb}$ and
$A(t)^{*\einhalb}$ coincide with $V$ as spaces and topologically with
constants independent of $t$, and that ${\mathcal
  A}(\cdot)^{\einhalb}$ is continuously differentiable with values in
${\mathcal L}(V,V')$. Note however that the property
$\DOMAIN(A(t)^{\einhalb}) = \DOMAIN(A(t)^{*\einhalb})$ is not always
true;  this equality is equivalent to the Kato's square
root property: $\DOMAIN(A(t)^{\einhalb}) = V$. The result of
\cite{Bar71} was extended in Arendt et
al. \cite{ArendtDierLaasriOuhabaz} by including the multiplication
$B(t)$ above and also weakening the regularity of ${\mathcal A}
(\cdot)^{\einhalb}$ from continuously differentiable to piecewise
Lipschitz. As in \cite{Bar71}, it is also assumed in
\cite{ArendtDierLaasriOuhabaz} that the domains of $A(t)^{\einhalb}$
and $A(t)^{*\einhalb}$ coincide with $V$ as spaces and topologically
with constants independent of $t$.

\noindent We emphasize that the above results from
\cite{ArendtDierLaasriOuhabaz, Bar71, Lions:book-PDE} on maximal
$L_2$--regularity do not give any information on maximal
$L_p$--regularity when $p \neq 2$ since the techniques used there are
based on a representation lemma in  (pre-) Hilbert  spaces.

\medskip

\noindent In the present paper we prove maximal $L_p$--regularity for
\eqref{eq:evol-eq} for all $p \in (1, \infty)$.  We extend the results
mentioned above and give a complete treatment of the problem with
initial data $u_0 \not= 0$ even when the forms are not necessarily
symmetric.  In particular, we obtain a positive answer to Problem~2
under even more general assumptions.

Our main result is the following.
\begin{theorem}\label{thm:main}
  Suppose that the forms $(\form{t})_{0 \le t \le \tau}$ satisfy the
  standing hypotheses
  \ref{item:constant-form-domain}--~\ref{item:uniform-accretivity} and
  the regularity condition
  \begin{equation}
    \label{eq:dini-regularity-of-forms}
   \left| \form{t}[u][v] - \form{s}[u][v] \right| \le \omega(|t{-}s|) \, \norm{u}_V \norm{v}_V
  \end{equation}
  where $\om: [0, \tau] \to [0, \infty)$ is a non-decreasing function such that 
\begin{equation}  \label{eq:Dini-3-2-condition}
    \int_0^\tau \tfrac{\omega(t)}{t^{\nicefrac32}}\,\dt < \infty. 
\end{equation}
Then the Cauchy problem (\ref{eq:evol-eq}) with $u_0 = 0$ has
maximal $L_p$--regularity in $H$ for all $p \in (1, \infty)$. If in addition $\om$
satisfies the $p$--Dini condition
\begin{equation}  \label{eq:p-Dini}
     \int_0^\tau \left(\tfrac{\omega(t)}{t} \right)^p \,\dt <  \infty,
\end{equation}
then (\ref{eq:evol-eq}) has maximal $L_p$--regularity for all $u_0 \in (H,
\DOMAIN(A(0)))_{1- \nicefrac1{p}, p}$. Moreover  there exists a positive constant $C$
such that
\[ 
\norm{u}_{L_p(0, \tau; H)} + \norm{u'}_{L_p(0, \tau; H)} + \norm{ A(\cdot) u(\cdot)}_{L_p(0, \tau; H)} \le C\left[ \norm{f}_{L_p(0, \tau; H)} + \norm{u_0}_{(H, \DOMAIN(A(0)))_{1-\nicefrac1{p}, p}} \right].
\] 
\end{theorem}

In this theorem, $(H, \DOMAIN(A(0)))_{1-\nicefrac1{p}, p}$ denotes the
classical real-interpolation space, see \cite[Chapter
1.13]{Triebel:interpolation} or \cite[Proposition 6.2]{Lunardi:book}.

We wish to point out that  Y. Yamamoto\footnote{see ``Solutions in {$L^p$} of abstract parabolic equations in {H}ilbert
  spaces'', J. Math. Kyoto Univ. \textbf{33} (1993), no.~2, 299--314.} states a result  which resembles to our previous theorem  in the setting of operators satisfying the 
  so-called Acquistapace-Terreni conditions on the corresponding resolvents. 
  Unfortunately it is difficult to follow his proofs.\footnote{It seems to be difficult to understand the proof of Lemma 4.3 at the beginning of page 306, the end of the proof of Proposition 2 at page 307, as well as estimates after formula (5.5) at page 310 in the proof of Lemma 3.4.}
  \medskip
  
  We have the following corollaries.

\begin{corollary}\label{cor1}
  Under the assumptions of the previous theorem, the Cauchy problem
  (\ref{eq:evol-eq}) with $u_0 = 0$ has maximal $L_2$--regularity in
  $H$.  If in addition $\om$ satisfies
\begin{equation}  \label{eq:2-Dini}
     \int_0^\tau \left(\tfrac{\omega(t)}{t} \right)^2 \,\dt <  \infty,
\end{equation}
then (\ref{eq:evol-eq}) has maximal $L_2$--regularity for all $u_0 \in \DOMAIN((\delta + A(0))^\einhalb)$.
In addition, there exists a positive constant $C$
such that
\[
\norm{u}_{W^1_2(0, \tau; H)} + \norm{ A(\cdot) u(\cdot)}_{L_2(0, \tau; H)} \le C\left[ \norm{f}_{L_2(0, \tau; H)} + 
\norm{u_0}_{ \DOMAIN((\delta + A(0))^\einhalb)} \right].
\]

\end{corollary}
Obviously, if $A(0)$ is accretive then $A(0)^\einhalb$ is well defined and $ \DOMAIN((\delta + A(0))^\einhalb) =  \DOMAIN(A(0)^\einhalb)$. 

This corollary solves Problem 2 even under more general conditions
than conjectured by J.L. Lions%

\begin{corollary}\label{coro:main}
  Assume that additionally to the standing assumptions
  \ref{item:constant-form-domain}--~\ref{item:uniform-accretivity}
  that the form $\form{t}[\cdot][\cdot]$ is piecewise $\al$--Hölder
  continuous for some $\al > \einhalb$. That is, there exist $t_0 = 0
  < t_1< ...< t_k = \tau$ such that on each interval $(t_i, t_{i+1})$
  the form is the restriction of a $\al$--Hölder continuous form on
  $[t_i, t_{i+1}]$. Assume in addition that at the discontinuity points, we have $\DOMAIN((\delta + A(t_j^-))^{\einhalb}) = 
  \DOMAIN((\delta + A(t_j^+))^{\einhalb})$. Then the Cauchy problem (\ref{eq:evol-eq}) has
  maximal $L_2$--regularity for all $u_0 \in \DOMAIN((\delta + A(0))^\einhalb)$
  and there exists a positive constant $C$ such that
\[
\norm{u}_{W^1_2(0, \tau; H)}  + \norm{ A(\cdot) u(\cdot)}_{L_2(0, \tau; H)} \le C\left[ \norm{f}_{L_2(0, \tau; H)} + 
\norm{u_0}_{ \DOMAIN((\delta + A(0))^\einhalb)} \right].
\]
\end{corollary}

In this corollary, $A(t_j^-)$ is the operator associated with the extension of the form at the left of point $t_j$. Similarly for 
$A(t_j^+)$. 

We mention that  Fuje and Tanabe \cite{FT73} constructed an evolution family associated with the non-autonomous problem considered here  when the 
form $\form{t}[\cdot][\cdot]$ is  $\al$--Hölder
  continuous for some $\al > \einhalb$. This is of independent interest but  it is not clear if one  obtains maximal regularity from 
  any property of  the corresponding evolution family. 

Now we explain briefly the strategy of the proof. A starting point is
a representation formula for the solution $u$ (recall that $u$ exists
in $V'$ by Lions theorem), which already appeared in the work of
Acquistapace and Terreni \cite{AcquistapaceTerreni}, namely
 \begin{equation}    \label{eq:AT00}
    \begin{split}
              u(t) = & \;  \int_0^t  e^{-(t{-}s)  \calA(t)} (\calA(t){-}\calA(s)) u(s)\,\ds \\ 
                     & \quad + \int_0^t e^{-(t{-}s) \calA(t) } f(s)\,\ds  + e^{-t \,\calA(t)}u_0.
    \end{split}
  \end{equation}
This allows us to  write $\calA(t) u(t) = (Q \calA(\cdot)u(\cdot))(t) + (L f)(t) + (Ru_0)(t)$,  
where
\begin{align*}
(Q g)(t) := & \; \int_0^t  \calA(t) e^{-(t{-}s) \calA(t)} (\calA(t) - \calA(s)) \, \calA(s)^{-1}g(s)\,\ds\\
(L g )(t) :=& \;  \calA(t) \int_0^t  e^{-(t{-}s) \calA(t)} g(s)\,\ds 
  \quad \text{and} \quad  
  (R u_0)(t) := \calA(t) e^{-t \, \calA(t)} u_0.
\end{align*}
Condition \eqref{eq:Dini-3-2-condition} allows us to prove
invertibility of the operator $(I{-}Q)$ on $L_p(0, \tau; H)$. The
operator $L$ is seen as a pseudo-differential operator with an
operator-valued symbol. We prove an $L_2$-boundedness result for such
operators in Section \ref{sec-pseudo}, see
Theorem~\ref{thmpseudo}. For operators with scalar-valued symbols,
this result is due to Muramatu and Nagase \cite{MuNa81}. We adapt
their arguments to our setting of operator-valued symbols.  This
theorem is then used to prove $L_2(0, \tau ; H)$-boundedness of
$L$. In order to extend $L$ to a bounded operator on $L_p(0, \tau ;
H)$, for $p \in (1, \infty)$, we look at $L$ as a singular integral
operator with an operator-valued kernel. We show that both $L$ and its
adjoint $L^*$ satisfy the well known Hörmander integral
condition. Finally, we treat the operator $R$ by taking the difference
with $\calA(0) e^{-t \,\calA(0)} u_0$ and using the functional calculus
for accretive operators on Hilbert spaces. In order to handle this
difference we use \eqref{eq:p-Dini}, the remaining term, $t \mapsto
A(0) e^{-t \, A(0)}u_0$ is then in $L_p(0, \tau; H)$ if and only if
$u_0 \in ( H, \DOMAIN(A(0)))_{1-\nicefrac{1}{p},p}$.
  
Although most of the arguments outlined here use heavily the fact that
$H$ is a Hilbert space, the strategy justifies some hope that the
results extend to other situations such as $L_q$--spaces. One would
then hope to prove $L_p(L_q)$ a priori estimates for parabolic
equations with time dependent coefficients.  In the last section of
this paper we give some applications and prove $L_p(L_2)$--a priori
estimates.

\bigskip

Maximal regularity may fail even for ordinary differential equations,
letting $H = \RR$.  We illustrate this by an example which is essentially   taken  from 
Batty, Chill and Srivastava \cite{BattyChillSrivastava}.

\begin{example}
  Consider $\varphi(t) = |t|^{-\nicefrac{1}p}$. Then $\varphi$ in $L_{q,
    \text{loc}}(\RR)$ for $1\le q < p$ but $\varphi \not\in
  L_p([0,\eps])$ for $\eps>0$. Chose a dense sequence $(t_n)$ of $[0,1]$ and a
  positive, summable sequence $(c_n)$. Define $a(t) := 1+ \sum c_n
  \varphi(t-t_n)$.  Then $a \in L_q([0,1])$ for $1 \le q < p$ but $a \not
  \in L_p(I)$ for any interval $I \subset [0,1]$. Consider the
  non-autonomous equation
  \begin{equation}    \label{eq:counterex-BCS}
\left\{
  \begin{array}{rcl}
x'(t) + a(t) x(t) &=& 1\\
x(0)&=&0
  \end{array}\right.
  \end{equation}
Then by variation of constants formula, $x(t) = \int_0^t \exp\bigl(-\int_s^t
a(r)\,\dr\bigr) \,\ds$. Since $a(r) \ge 0$,
\begin{align*}
|a(t) x(t)| = & \; a(t)  \int_0^t \exp\Bigl(- \int_s^t a(r)\,\dr\Bigr) \,\ds\\
          \ge & \; a(t)  \int_0^t \exp\Bigl(- \int_0^1 a(r)\,\dr \Bigr) \,\ds
           =   \; C t\, \, a(t).  
\end{align*}
Therefore, for $0<\alpha< \beta \le 1$ we have $|a(t)x(t)| \ge \alpha C\; a(t)$ on
$[\alpha, \beta]$ which implies that (\ref{eq:counterex-BCS}) cannot have
maximal $L_p$--regularity. 

On the other hand, if we replace the constant function $1$ by 
$f$ we obtain
\begin{align*}
   |a(t) x(t)| 
=   & \; a(t) | \int_0^t  f(s) \exp\Bigl(- \int_s^t a(r)\,\dr\Bigr)| \,\ds\\
\le & \; a(t) \int_0^t  | f(s) | \, \ds  \le C a(t) \| f\|_{L_q}
\end{align*}
on $[0,1]$ and this shows that (\ref{eq:counterex-BCS}) has maximal
$L_q$-regularity for $q<p$. 

Notice however, that letting $p{=}2$ this
example is not a counterexample to the questions we raise, since our
standing hypothesis \ref{item:uniform-continuity} is not satisfied
here.

\noindent Observe also that the operators in this example are all bounded and
commute. Thus, these last two properties are not enough to obtain
maximal $L_2$-regularity for non-autonomous evolution equations.
\end{example}

\section{Preparatory lemmas}
In this section we prove most of the main arguments which we will need
for the proofs of our  results.  The only missing argument here
concerns boundedness of pseudo-differential operators with
operator-valued symbols which we write in a separate section for
clarity of exposition.  We formulate our arguments in a series of
lemmata.

 Throughout this section we will suppose that
~\ref{item:constant-form-domain}--~\ref{item:uniform-accretivity} are
satisfied. Let $\mu \in \RR$ and set $v(t) := e^{-\mu t} u(t)$. If $u$
satisfies \eqref{eq:evol-eq}, then $v$ satisfies the evolution equation
\[
\left\{
  \begin{array}{rcl}
     v'(t) + (\mu + A(t))\,v(t) &=& f(t)e^{-\mu t} \\
     u(0)&=&u_0. 
  \end{array}
\right.
\]
In addition, $v \in W^{1}_p(0, \tau; H)$ {\it if and only if} $u \in
W^{1}_p(0, \tau; H)$. This shows that we may replace $\calA(t)$
(resp. $A(t)$) by $\calA(t) + \mu$ (resp. $A(t) + \mu$). Therefore, we
may suppose without loss of generality that $\delta{=}0$ in
\ref{item:uniform-accretivity}. In particular, we may suppose that
$A(t)$ and $\calA(t)$ are boundedly invertible by choosing $\mu > 0$
large enough.  We will do so in the sequel without further mentioning
it.

It is known that ${-}A(t)$ generates a bounded holomorphic semigroup
on $H$.  The same is true for ${-}\calA(t)$ on $V'$. We write this
explicitly in the next proposition in order to point out that the
constants involved in the estimates are uniform with respect to
$t$. The arguments are standard and can be found e.g. in \cite[Section
1.4]{Ouhabaz:book}.  Denote by $S_\theta$ the open sector
$\sector{\theta} = \{ z \in \CC^* \SUCHTHAT |\text{arg}(z) < \theta \}$
with vertex $0$.

\begin{proposition}\label{prop:sg-extrapolation} 
  For any $t \in [0, \tau]$, the operators $-A(t)$ and $-\calA(t)$
  generate strongly continuous analytic semigroups of angle
  $\pihalbe{-}\arctan(\nicefrac{M}\al)$ on $H$ and $V'$, respectively. In
  addition, there exist constants $C$ and $C_\theta$, independent of
  $t$, such that
  \begin{enumerate}
  \item\label{item:rien} 
     $\displaystyle   \norm{ e^{-z \, A(t) } }_{\BOUNDED(H)}  \le 1\ 
      \text{and}\ 
      \norm{ e^{-z \,  \calA(t) } }_{\BOUNDED(V')}  \le C \quad\text{for all}\quad z \in \sector{\pihalbe{-}\arctan(\nicefrac{M}{\alpha})}$,
  \item \label{item:standard-0}
     $\displaystyle   \norm{ A(t) e^{-s \,  A(t) } }_{\BOUNDED(H)}  \le \tfrac{C}{s}
        \quad\text{and}\quad
       \norm{ \calA(t) e^{-s \,  \calA(t) } }_{\BOUNDED(V')}  \le \tfrac{C}{s}$, 
  \item \label{item:standard-1} $\norm{  e^{-s \,  A(t)} x}_V \le \tfrac{C}{\sqrt{s}} \norm{x}_H \quad\text{and}\quad 
  \norm{  e^{-s \,  \calA(t)} \phi}_H \le \tfrac{C}{\sqrt{s}} \norm{\phi}_{V'}$, 
    \item \label{item:resol}
      $\norm{ ( z - A(t) )^{-1} x}_V \le \tfrac{C_\theta}{\sqrt{|z|}} \norm{x}_H \ \text{for}\  z \notin \sector{\theta}
         \quad  \text{and  fixed }\quad \theta > \arctan(\nicefrac{M}{\alpha}).$
   \item\label{item:tout}
  All the previous estimates hold for $A(t) + \mu$ with constants independent of $\mu$ for  $\mu \ge 0$.
  \end{enumerate}
\end{proposition}
\begin{proof}
  Fix $t \in [0, \tau]$. By uniform boundedness and coercivity,
  \begin{equation}\label{num0}
| \Im  \form{t}[u][u] |  \  \le M \norm{u}_V^2 \le \nicefrac{M}{\alpha} \, \Re \form{t}[u][u].
\end{equation}
This means that $A(t)$ has numerical range contained in the closed
sector $\overline{\sector{\omega_0}}$ with $\omega_0 = \arctan(\nicefrac{M}{\alpha})$. This
implies the first part of assertion \ref{item:rien}, see
e.g. \cite[Theorem~1.54]{Ouhabaz:book}.  Let $u \in V$ and set
$\varphi = (\la + \calA(t)) u \in V'$.  Then
\[  
   \dprod{\varphi}{u} = \la \norm{u}_H^2 + \form{t}[u][u]
\]
and so coercivity yields
  \begin{equation}    \label{eq:extrapolation-1}
    \begin{split}
   \norm{u}_V^2 & \le \tfrac1{\al}  \Re \form{t}[u][u] 
                 \le  \tfrac1{\al} \bigl( |\dprod{\varphi}{u}| + |\la| \norm{u}_H^2\bigr)  \\
               & \le  \tfrac1{\al} \bigl( \norm{\varphi}_{V'} \norm{u}_V  + |\la| \norm{u}_H^2\bigr). 
    \end{split}
  \end{equation}
  We aim to estimate $ |\la| \norm{u}_H^2$ against $\norm{u}_V
  \norm{\varphi}_{V'}$.  Since the numerical range of $\form{t}$ in
  contained in $\overline{\sector{\omega_0}}$ we have
\begin{align*}
  \dist(\la, -\sector{\om_0}) \, \norm{u}_H^2 & \le \; \Bigl|\la + \form{t}[\tfrac{u}{\norm{u}}][\tfrac{u}{\norm{u}}] \Bigr| \,\norm{u}_H^2\\
                                    & \le \; \bigl| \dprod{ (\la I + \calA(t))u}{u}\bigr| 
                                      \le \; \norm{u}_V \norm{\varphi}_{V'}.
\end{align*}
Now let $\theta>\om_0$ and $\la \not \in \sector{\theta}$. Then $
\dist(\la, -\sector{\om_0}) \ge |\la| \sin(\theta{-}\om_0)$ and therefore
\[
   |\la| \norm{u}_H^2 \le \tfrac{1}{\sin(\theta{-}\om_0)}  \norm{u}_V \norm{\varphi}_{V'}
\]
as desired. From this and (\ref{eq:extrapolation-1}) it follows that
\begin{equation}  \label{eq:boundedly-invertible}
  \norm{u}_V \le  \tfrac1{\al} \bigl(1+ \tfrac{1}{\sin(\theta{-}\om_0)}\bigr)\,  \norm{(\la + \calA(t)) u  }_{V'}
\end{equation}
uniformly for all $\la \not \in \sector{\theta}$, $\theta>\om_0$. This
implies that $(\la +\calA(t))$ is invertible with a uniform norm bound
on $\sector{\theta}^\complement$, $\theta>\om_0$. This is equivalent to
$-\calA(t)$ being the generator of a bounded analytic semigroup on
$V'$. The bound is independent of $t$. This proves assertion
\ref{item:rien}.

Assertion \ref{item:standard-0} follows from the analyticity of the
semigroups $(e^{-s \, A(t)})_{s\ge 0}$ on $H$ and $(e^{-s \,
  \calA(t)})_{s\ge 0}$ on $V'$ and the Cauchy formula as usual.

For assertion \ref{item:standard-1},  observe that for $x \in H$
\begin{align*}
 \al \norm{  e^{-s \, A(t)} x}_V^2 
\le & \;  \Re \form{t}[ e^{-s \, \calA(t)} x][e^{-s \, \calA(t)} x]  \\
  = & \;  \Re \sprod{ \calA(t) e^{-s \, \calA(t)} x}{e^{-s \,\calA(t)} x} \\
\le &\;   \tfrac{C}{s}\, \norm{x}_H^2.
\end{align*}
The second inequality in \ref{item:standard-1} follows by duality.

The estimate \ref{item:resol} follows in a a natural way from
\ref{item:rien} and \ref{item:standard-1} by writing the resolvent as
the Laplace transform of the semigroup.  Finally, in order to prove
assertion \ref{item:tout}, we notice that for a constant $\mu \ge 0$ we
have
\[
    \norm{e^{-z \,(A(t) + \mu)} x }_H \le \norm{e^{-z \,  A(t)} x }_H,
\]
for all $z \in \sector{\pihalbe{-}\arctan(\nicefrac{M}{\alpha})}$. The same
estimate also holds when replacing the norm of $H$ by the norm of $V$
or the norm of $V'$. Now we use the Cauchy formula to obtain
\ref{item:standard-0} for $A(t) + \mu$.  Assertion \ref{item:resol} for
$A(t) + \mu$ follows again by the Laplace transform. The other
estimates are obvious.
\end{proof}

Finally we mention the following easy corollary of the proposition. 
\begin{corollary}  \label{cor:resolvent-diffs}
Let $\omega: \RR \to \RR_+$ be some function  and 
assume that
\[
| \form{t}[u][v]-\form{s}[u][v]| \le \omega(|t{-}s|) \norm{u}_V \norm{v}_V
\]
for all $u, v  \in V$. Then 
\[
   \norm{ R(z,A(t)) - R(z,A(s)) }_{\BOUNDED(H)} \le \tfrac{ c_\theta}{|z|} \omega(| t{-}s|)
\]
for all $z \notin \sector{\theta}$ with any fixed  $\theta > \arctan(\nicefrac{M}{\alpha})$.
\end{corollary}
\begin{proof} Observe that for $u, v \in V$,
\begin{align*}
    & \;\bigl|\sprod{ R(z,A(t)) u - R(z,A(s)) u}{v}\bigr| \\
  = &\; \bigl|\sprod{ R(z,A(t)) ( A(s)-A(t) ) R(z,A(s)) u}{v}\bigr| \\
  = &\; \bigl|\sprod{ A(s) R(z,A(s)) u} { R(z,A(t))^* v} - \sprod{ A(t)R(z,A(s))u} { R(z,A(t))^* v}\bigr| \\
  = &\; \bigl|\form{s} [R(z,A(s)) u][ R(z,A(t))^* v] - \form{t}[R(z,A(s)) u][R(z,A(t))^* v]\bigr|\\
\le & \; \tfrac{c_\theta} {|z|} \, \omega(| s{-}t |) \, \norm{u}_H\,\norm{v}_H,
\end{align*}
where we used Proposition~\ref{prop:sg-extrapolation}\ref{item:resol}.
  \end{proof}

 Next  we come to a formula for the solution $u$ of (P') in $V'$.
  Recall that $u$ exists by Lions' theorem mentioned in the
  introduction.  
 This formula already appears in Acquistapace-Terreni
  \cite{AcquistapaceTerreni}.  
  Fix $f \in C_c^\infty(0, \tau; H)$ and $u_0 \in H$.

\begin{lemma}\label{lem:AT-formula-for-u}
  For almost every  $t \in (0, \tau)$, we have, in $V'$,
  \begin{equation}    \label{eq:AT-formula-for-u}
    \begin{split}
              u(t) = & \;  \int_0^t  e^{-(t{-}s) \calA(t)} (\calA(t){-}\calA(s)) u(s)\,\ds \\ 
                     & \quad + \int_0^t e^{-(t{-}s)  \calA(t) } f(s)\,\ds  + e^{-t \, \calA(t)}u_0.
    \end{split}
  \end{equation}
\end{lemma}
\begin{proof}
 Recall that $u \in W^1_2(0, \tau; V')$ by Lions' theorem and hence $u$ has a continuous representative. 
 Fix $t \in (0, \tau)$ such that $\DOMAIN(\calA(t)) = V$ (recall that this is true  for almost all $t$). 
 Set $v(s) := e^{-(t{-}s) \, \calA(t) } u(s)$ for $0<s\le
  t<\tau$. Recall that $-\calA(t)$ generates a bounded semigroup $e^{-s  \calA(t) }$  on $V'$, see
  Proposition~\ref{prop:sg-extrapolation}~\ref{item:standard-0}.  
  Since  $u \in W^1_2(0,\tau; V')$,   $v$ has a distributional derivative in $V'$ which satisfies 
  \begin{align*}
    v'(s) = & \; \calA(t) e^{-(t{-}s)  \calA(t) } u(s) + e^{-(t{-}s)  \calA(t) } (f(s) - \calA(s) u(s) )\\
          = & \;  e^{-(t{-}s)  \calA(t) } (\calA(t){-}\calA(s)) u(s) + e^{-(t{-}s)  \calA(t) } f(s). 
  \end{align*}
  Using the fact that $u \in L_2(0, \tau; V)$, it follows that $v \in W^1_2(0,\tau; V')$. Hence
\[
  v(t) - v(0) =  \int_0^t v'(s) \, \ds = \int_0^t  e^{-(t{-}s) \calA(t) } (\calA(t){-}\calA(s)) u(s)\,\ds + \int_0^t e^{-(t{-}s) \calA(t) } f(s)\,\ds. 
\]
This gives
(\ref{eq:AT-formula-for-u}) by observing that $v(t) = u(t)$ and
$v(0)=e^{-t \,\calA(t)}u_0$.
\end{proof}

\begin{lemma}\label{lem:AT-formula-for-Au}
  Suppose (\ref{eq:Dini-3-2-condition}). Then for almost all $t \in [0,  \tau]$
\[
\calA(t) u(t) = (Q \calA(\cdot)u(\cdot))(t) + (L f)(t) + (Ru_0)(t)
\]
in $V'$, where
\begin{equation}  \label{eq:operator-Q}
(Q g)(t) := \int_0^t  \calA(t) e^{-(t{-}s)  \calA(t)} (\calA(t) - \calA(s)) \calA(s)^{-1}g(s)\,\ds
\end{equation}
and
\begin{equation}  \label{eq:operator-L}
  (L f )(t) :=   \calA(t) \int_0^t  e^{-(t{-}s)  \,\calA(t)} f(s)\,\ds 
  \quad \text{and} \quad  
  (R u_0)(t) := \calA(t) e^{-t \, \calA(t)} u_0.
\end{equation}
\end{lemma}
\begin{proof}
 As in the  proof of the previous lemma, we  fix $t \in (0, \tau)$ such that $V = \DOMAIN(\calA(t))$.  
  It is enough to prove that each term in the sum
  (\ref{eq:AT-formula-for-u}) is in $V$. Observe
  that by analyticity, $e^{-t \,\calA(t)}u_0 \in \DOMAIN(\calA(t)) =
  V$. In passing we also note that since $u_0 \in H$, $\calA(t) e^{-t
  \,  \calA(t)}u_0 = A(t) e^{-t \,A(t)}u_0$.
  
   Concerning the first term,
  we recall that $u(s) \in V$ for almost all
  $s$ (note that $u \in L_2(0, \tau; V)$). Therefore,
  $e^{-(t{-}s) \calA(t)} (\calA(t){-}\calA(s)) u(s) \in V$  for almost every $s < t$ by the analyticity of the semigroup generated by
  $-\calA(t)$. In addition,
\begin{align*}
  \norm{ \calA(t) e^{(t{-}s) \calA(t)} (\calA(t){-}\calA(s)) u(s) }_{V'}
\le &\; \tfrac{C}{t-s} \norm{(\calA(t){-}\calA(s)) u(s) }_{V'} \\
 =  &\; \tfrac{C}{t-s} \sup_{\norm{v}_V = 1} |\form{t}[u(s)][v] - \form{s}[u(s)][v]|\\
\le &\; \tfrac{C}{t-s} \, \omega(t{-}s) \norm{u(s)}_V.
\end{align*}
Note that the operator
\begin{equation}\label{op.H}
{\mathcal H}: \ h \mapsto \int_0^t \tfrac{\omega(t{-}s)}{t-s}  \,h(s) \, \ds 
\end{equation}
is bounded on $L_p(0, \tau; \RR)$ for all $p \in [1, \infty]$.  The
reason is that the associated kernel $(t,s) \mapsto \eins_{[0,
  t]}(s)\tfrac{\omega(t{-}s)}{t-s}$ is integrable with respect to each
variable with a uniform bound  with respect to the other variable as can be seen easily from
(\ref{eq:Dini-3-2-condition}).  
Recall again that $\norm{u(\cdot)}_V \in L_2(0, \tau; \RR)$. Hence
\[
s \mapsto  \eins_{[0, t]}(s) \cdot \calA(t) e^{-(t{-}s)  \,\calA(t)} (\calA(t){-}\calA(s)) u(s)
\]
is in $L_1(0, \tau;V')$. Therefore, for every $\varepsilon > 0$
\[
   \int_0^{t-\varepsilon}  e^{-(t{-}s)  \calA(t)} (\calA(t){-}\calA(s)) u(s)\,\ds  \in \DOMAIN(\calA(t))
\]
and the fact that  $\calA(t)$ is a closed operator gives
\[
\calA(t)    \int_0^{t-\varepsilon}  e^{-(t{-}s) \calA(t)} (\calA(t){-}\calA(s)) u(s)\,\ds
=    \int_0^{t-\varepsilon} \calA(t)  e^{-(t{-}s) \calA(t)} (\calA(t){-}\calA(s)) u(s)\,\ds.
\]
 Since 
\[ 
  s \mapsto  \eins_{[0, t]}(s) \cdot \calA(t) e^{-(t{-}s) \calA(t)} (\calA(t){-}\calA(s)) u(s) \in L_1(0, \tau;V'),
\]
 we may let $\varepsilon \to 0$ and obtain 
$\int_0^{t}  e^{-(t{-}s) \calA(t)} (\calA(t){-}\calA(s)) u(s)\,\ds  \in \DOMAIN(\calA(t))$ and 
\[
\calA(t)    \int_0^t  e^{-(t{-}s) \calA(t)} (\calA(t){-}\calA(s)) u(s)\,\ds
=    \int_0^{t} \calA(t)  e^{-(t{-}s)  \calA(t)} (\calA(t){-}\calA(s)) u(s)\,\ds.
\]
Finally, the equality (\ref{eq:AT-formula-for-u}) and the fact that
$u(t) \in V$ for almost all $t$ yields
\[
\int_0^t e^{-(t{-}s) \calA(t)} f(s) \, \ds \in \DOMAIN(\calA(t)).
\]
This proves the lemma.
\end{proof}

\medskip

Recall the definition of the operator $L$,
\[
   L f(t) = \calA(t) \int_0^t e^{-(t{-}s) A(t)} f(s)\,\ds. 
\]
Let $f \in C_c^\infty(0, \tau; H)$ and denote by $f_0$ its extension
to the whole of $\RR$ by $0$ outside $(0, \tau)$.  Observe that $f_0$
is then in the Schwarz class $\SCHWARTZ(\RR; H)$.  We denote for
Fourier transform of $f_0$ by $\FOURIER f_0$ or $\widehat
f_0$. Clearly, 
\[
 \int_{-\infty}^t e^{-(t{-}s) A(t)} f_0(s)\,\ds 
 =  \tfrac1{2\pi} \;  \int_{-\infty}^t e^{-(t{-}s) A(t)} \int_\RR e^{\ui s \xi } \widehat{f_0}(\xi)\,\dxi \,\ds
\]
Now, exponential stability of $(e^{-s \, A(t)})_{s\ge 0}$ and the
fact that $f_0 \in \SCHWARTZ(\RR; H)$ allows us to use Fubini's theorem,
giving
\begin{align*}
 \int_{-\infty}^t e^{-(t{-}s) A(t)} \int_\RR e^{\ui s \xi } \widehat{f_0}(\xi)\,\dxi\,\ds
 = & \;   \int_\RR  \Bigl(\int_{-\infty}^t e^{-(t{-}s) (\ui\xi + A(t))}  \,\ds \Bigr) \widehat{f_0}(\xi)  e^{\ui t \xi }\,\dxi \\
 = & \;   \int_\RR (\ui\xi+A(t))^{-1}\widehat{f_0}(\xi)  e^{\ui t \xi }\,\dxi.
\end{align*}
It follows  that 
\begin{equation}\label{egal0}
  \int_{-\infty}^t e^{-(t{-}s) A(t)} f_0(s)\,\ds  =  \tfrac1{2\pi} \int_\RR (\ui\xi+A(t))^{-1}\widehat{f_0}(\xi)  e^{\ui t \xi }\,\dxi.
\end{equation}
Observe that the right hand side of (\ref{egal0}) converges in norm (as a Bochner
integral) and that the same holds for
\[
  \int_\RR A(t) (\ui\xi+A(t))^{-1}\widehat{f_0}(\xi)  e^{\ui t \xi }\,\dxi
\]
since $\widehat{f_0} \in \SCHWARTZ(\RR; H)$. Thus, both  terms in \eqref{egal0} 
take  values in $\DOMAIN(A(t))$. This shows that for $f\in
C_c^\infty(0, \tau; H)$, $(Lf)(t)$ is a well-defined function taking
pointwise values in $H$. Hille's theorem (see e.g.
\cite[II.2,~Theorem~6]{DiestelUhl:vector-measures}) then allows us to
take the closed operator $A(t)$ inside the integral which finally
gives the representation formula
\begin{equation}  \label{eq:L-comme-pso}
    L f(t) = \FOURIER^{-1}\bigl(\xi \mapsto  \sigma(t, \xi)  \widehat{f_0}(\xi) \bigr)(t),
\end{equation}
that allows us to see $L$ as a pseudo-differential operator with
operator-valued symbol
\begin{equation}  \label{eq:definition-symbol-sigma}
  \sigma(t, \xi) = 
\left\{
  \begin{array}{lcl}
     A(0)(\ui\xi+A(0))^{-1} & \text{if} &  t<0  \\
     A(t)(\ui\xi+A(t))^{-1} & \text{if} &  0\le t \le \tau \\
     A(\tau)(\ui\xi+A(\tau))^{-1} & \text{if} &  t > \tau \\
  \end{array}\right.
\end{equation}

\begin{lemma}\label{lem:operateur-L-sur-L2}
  Suppose that in addition to our standing assumptions
  \ref{item:constant-form-domain}-~\ref{item:uniform-accretivity} that 
  \eqref{eq:dini-regularity-of-forms} holds 
  with $\om: [0, \tau] \to [0, \infty)$  a non-decreasing function such that 
\begin{equation}  \label{eq:Dini-2-1-condition}
    \int_0^\tau \tfrac{\omega(t)^2}{t}\,\dt < \infty. 
\end{equation}
 Then $L$ is a bounded operator on $L_2(0, \tau; H)$.  
\end{lemma}
\begin{proof}
  We prove the Lemma by verifying the conditions of Theorem~\ref{thmpseudo} below. 
  Let $\sigma(t, \xi)$ be given by (\ref{eq:definition-symbol-sigma}). 
 We need to prove  that 
  \begin{align}
    \label{eq:yamazaki:1}
    \bignorm{ \partial_\xi^k \, \sigma(t, \xi) }_{\BOUNDED(H)} 
       & \le \; c \tfrac{1}{{(1+\xi^2)}^{\nicefrac{k}2}},\\
    \label{eq:yamazaki:2}
    \bignorm{ \partial_\xi^k \, \sigma(t, \xi) - \partial_\xi^k \, \sigma(s, \xi) }_{\BOUNDED(H)} 
       & \le \; c \tfrac{\omega(t{-}s)}{{(1+\xi^2)}^{\nicefrac{k}2}}, 
  \end{align}
  for $k = 0, 1, 2$.  For $k=0$, (\ref{eq:yamazaki:1}) is just the
  sectoriality of $A(t)$, see Proposition~\ref{prop:sg-extrapolation}
  whereas (\ref{eq:yamazaki:2}) is precisely
  Corollary~\ref{cor:resolvent-diffs}. Observe that a holomorphic
  function that satisfies 
\[
   \norm{ f(z) } \le C \; \tfrac{1}{|z|}
\]
on  the complement of a sector of angle $\theta$ will automatically satisfy 
\[
   \norm{ f^{(k)} (z) } \le C_{\theta, k} \, C \; \tfrac{1}{|z|^{k+1}}
\]
on the complement of strictly larger sectors, simply by Cauchy's integral
formula for derivatives. Conditions (\ref{eq:yamazaki:1}) and
(\ref{eq:yamazaki:2}) follow therefore for all $k \ge 1$.
\end{proof}

Next we prove that the operator $L$ extends to a  bounded operator on $L_p(0,\tau;H)$.
\begin{lemma}\label{lemLp}
Under the assumptions of the previous lemma the operator $L$ is  bounded  on $L_p(0, \tau; H)$ for all
$p \in (1, \infty)$.
\end{lemma}
\begin{proof} The operator $L$ is  a singular integral operator with operator-valued kernel
\[
   K(t,s) = \mathbbm{1}_{\{0\le s \le t \le \tau\}} A(t) e^{-(t-s)A(t)},
\]
where $\mathbbm{1}$ denotes the indicator function. We prove that both
$L$ and $L^*$ are of weak type $(1,1)$ operators and we conclude by
the Marcinkiewicz interpolation theorem together with the previous
lemma that $L$ is bounded on $L_p(0, \tau ; H)$ for all $p \in (1,
\infty)$. It is known (see, e.g. \cite[Theorems
III.1.2 and III.1.3]{RRT86}) that $L$ (respectively $L^*$) is of
weak type $(1,1)$ if the corresponding kernel $K(t, s)$ satisfies the  Hörmander
integral condition. This means that we have to verify 
\begin{equation}\label{Horm1}
 \int_{|t-s|\ge 2|s'-s|} \norm{ K(t,s) - K(t,s') }_{\BOUNDED(H)} \,\dt \le C
\end{equation}
and 
\begin{equation}\label{Horm2}
\int_{|t-s|\ge 2|s'-s|} \norm{ K(s,t) - K(s',t) }_{\BOUNDED(H)} \,\dt \le C
\end{equation}
for some constant $C$ independent of $s, s' \in (0, \tau)$.  Note that
the above mentioned theorems in \cite{RRT86} are formulated for
integral operators on $L_p(\RR; H)$ instead of $L_p(0, \tau; H)$;
however it is known that Hörmander's integral condition works on any
space satisfying the volume doubling condition, see \cite[page
15]{RRT86}.

First consider the integral in \eqref{Horm1}. When $s \le s'$ and
$2|s'{-}s| > \tau$ the integral vanishes. When $0 \le s \le s' \le
t\le \tau$ and $2 s' {-}s \le \tau$, using that the semigroup
$(e^{-s \,A(t)})_{s\ge 0}$ generated by $-A(t)$ is bounded
holomorphic, with a norm bound independent of $t$, we have for some
constant $C$
\begin{align*}
   &\; \int_{|t{-}s|\ge 2|s'{-}s|} \norm{ K(t,s) - K(t,s') }_{\BOUNDED(H)} \,\dt \\
 = &\; \int_{2s'{-}s}^\tau \norm{ A(t) e^{-(t{-}s)A(t)} - A(t) e^{-(t{-}s')A(t)}  }_{\BOUNDED(H)} \, \dt\\
 = &\;  \int_{2s'{-}s}^\tau \int_s^{s'} \norm{ A(t)^2 e^{-(t{-}r)A(t)} \,\dr  }_{\BOUNDED(H)} \, \dt\\
\le&\; C \int_{2s'{-}s}^\tau \int_s^{s'} \tfrac{1}{(t{-}r)^2} \,\dr \,\dt = C \int_{2s'{-}s}^\tau \left[ \tfrac{1}{t{-}s'} - \tfrac{1}{t{-}s} \right] \, \dt\\
 = &\; C \left[ \log \tfrac{t{-}s'}{t{-}s\,} \right]_{t = 2s'{-}s}^{t= \tau} \le C \log 2.
\end{align*}
When $s' < s$ and $3s-2s'>\tau$, the integral (\ref{Horm1})
vanishes. When $s'<s$ and $3s-2s' < \tau$, a similar calculation to
the above shows that the integral is bounded by $C
\log(\nicefrac{3}{2})$.

\noindent We now consider \eqref{Horm2}. When $s\le s'$, as above, we
may assume that $3s{-}2s'>0$, since otherwise the integral in
\eqref{Horm2} vanishes. We have
\begin{align*}
   &\; \int_{|t{-}s|\ge 2|s'{-}s|} \norm{ K(s,t) - K(s',t) }_{\BOUNDED(H)} \dt \\
=  &\; \int_{0}^{3s{-}2s'}  \norm{ A(s) e^{-(s{{-}}t)  A(s)} - A(s') e^{-(s'{-}t)  A(s')}  }_{\BOUNDED(H)} \,\dt\\
\le&\; \int_{0}^{3s{-}2s'}  \norm{ A(s) e^{-(s{-}t)  A(s)} - A(s) e^{-(s'{-}t)  A(s)}  }_{\BOUNDED(H)} \,\dt \\
   &\quad +  \int_{0}^{3s{-}2s'}  \norm{ A(s) e^{-(s'{-}t)  A(s)} - A(s') e^{-(s'{-}t)  A(s')}  }_{\BOUNDED(H)} \,\dt =: I_1 + I_2.
\end{align*}
The first term $I_1$ is handled exactly as in the proof of
\eqref{Horm1}. For the second term $I_2$, we write by the functional
calculus
\[
  A(s) e^{-(s'{-}t)  A(s)} - A(s') e^{-(s'{-}t) A(s')} 
 = \tfrac1{2\upi  i} \int_\Gamma  z e^{-t z} \bigl[ R(z, A(s)) - R(z, A(s'))\bigr]  \,\dz
\]
where $\Gamma$ is the boundary of an appropriate sector
$\sector{\theta}$.  We apply  Corollary~\ref{cor:resolvent-diffs} to deduce
\begin{align*}
\norm{ A(s) e^{-(s'{-}t)  A(s)} - A(s') e^{-(s'{-}t)  A(s')}  }_{\BOUNDED(H)} 
\le & \; C \int_0^\infty r e^{-(s'{-}t)r \cos\theta} \tfrac{\omega(s'{-}s)}{r} \,\dr\\
\le & \; C \, \tfrac{\omega(s'{-}s)}{s'{-}t}.
\end{align*}
Therefore,
\begin{align*}
    &   \int_{0}^{3s{-}2s'}  \norm{ A(s) e^{-(s'{-}t) A(s)} - A(s') e^{-(s'{-}t)  A(s')}  }_{\BOUNDED(H)} \,\dt  \\
\le & \; C \int_0^{3s{-}2s'} \tfrac{\omega(s'{-}s)}{s'{-}t} \,\dt
\le  \; C \int_0^\tau \omega(r) \,\tfrac{\dr}{r} = C',
\end{align*}
where we used the fact that $\omega$ is non-decreasing and $s'{-}s \le
s'{-}t$ to write the second inequality. Finally, the integral
(\ref{Horm2}) in the case $s'<s$ is treated similarly. 
Remark: A similar
reasoning for the weak type $(1,1)$ estimate for $L$ and $L^*$ appears
in \cite[p. 1051]{HM99}.
\end{proof}

Now we study the operator $R$. 
\begin{lemma}\label{lem:operateur-R}
  Assume (\ref{eq:p-Dini}). Then there exists $C > 0$ such that for
  every $u_0 \in ( H, \DOMAIN(A(0)))_{1-\nicefrac{1}{p},p}$:
\[ 
\norm{Ru_0}_{L_p(0, \tau;H)} \le  C \norm{u_0}_{( H, \DOMAIN(A(0)))_{1-\nicefrac{1}{p},p}}.
\]
\end{lemma}

\begin{proof}
  Recall that the operator $R$ is given by $(R g)(t) = A(t) e^{-t \,A(t)} g$ for $g \in H$.  Let
\[
   (R_0 g)(t) := A(0) e^{-t \,  A(0)}g.
\]
We aim to estimate the difference $(R-R_0)g$.  Let $\Gamma = \partial
\sector{\theta}$ with $\theta \in (\om_0, \pihalbe)$ and $\om_0$ is as in the proof of Proposition \ref{prop:sg-extrapolation}.
 Then, for $v \in
V$, the functional calculus for the sectorial operators $A(t)$ and
$A(0)$ gives
\begin{align*}
   & \;  \dprod{ A(t)  e^{-t \,  A(t)} g -  A(0)  e^{-t \,  A(0)}g }{v}\\
 = & \;  \tfrac1{2\upi  i} \int_\Ga  \dprod{ z e^{-t z} \bigl[ R(z, A(t)) - R(z, A(0))\bigr] g }{v} \, \dz\\
 = & \;  \tfrac1{2\upi  i} \int_\Ga  \dprod{ z e^{-t z} R(z, \calA(t)) \bigl[\calA(0)  - \calA(t)\bigr] R(z, A(0))g }{v} \, \dz\\
 = & \;  \tfrac1{2\upi  i} \int_\Ga  \dprod{ z e^{-t z} \bigl[\calA(0)  - \calA(t)\bigr] R(z, A(0))g }{R(z, A(t))^*v} \, \dz\\
 = & \;  \tfrac1{2\upi  i} \int_\Ga z e^{-t z} \form{0}[  R(z, A(0))g ] [ R(z, A(t))^* v ] - \form{t}[  R(z, A(0))g ] [ R(z, A(t))^*v ] \,\dz.
\end{align*}
Now, taking the absolute value it follows from
Proposition~\ref{prop:sg-extrapolation}\ref{item:resol} that
\begin{align*}
 \left|\dprod{ (Rg - R_0 g)(t)}{v} \right|
\le &\;  \tfrac{C_\al}{2\upi } \int_\Ga \omega(t)  |z| e^{-t \,\Re(z)} \norm{ R(z, A(0))g }_V \norm{ R(z, A(t))^*v }_V  \,|\dz|\\
\le &\;  \tfrac{C_{\al, \theta}}{2\upi }  \omega(t) \norm{g}_H \norm{v}_H  \int_\Ga   e^{-t \,\Re z} \,|\dz| \\
\le&\,   C'  \tfrac{\omega(t)}{t}  \norm{g}_H \norm{v}_H.
\end{align*}
Since this true for all $v \in H$ we conclude that 
\begin{equation}\label{RR}
\norm{(R u_0)(t)  - (R_0 u_0)(t)}_H \le C'  \tfrac{\omega(t)}{t}  \norm{u_0}_H.
\end{equation}
From the hypothesis (\ref{eq:p-Dini}) it follows that $Ru_0- R_0 u_0
\in L_p(0, \tau; H)$.  On the other hand, since $A(0)$ is invertible,
it is well-known that $A(0) e^{-t \,A(0)}u_0 \in L_p(0, \tau; H)$ if and
only if $u_0 \in ( H, \DOMAIN(A(0)))_{1-\nicefrac{1}{p},p}$ (see Triebel
\cite[Theorem~1.14]{Triebel:interpolation}). Therefore, $Ru_0 \in
L_p(0, \tau;H)$ and the lemma is proved.
  \end{proof}

\section{Proofs of the main results}\label{sec:main}
\begin{proof}[Proof of Theorem~\ref{thm:main}]
  Assume first that $u_0 = 0$ and let $f \in C_c^\infty(0, \tau; H)$. From
  Lemma~\ref{lem:AT-formula-for-Au} it is clear that
  \begin{equation}    \label{eq:intermediate}
   (I-Q) A(\cdot)u(\cdot) = L f(\cdot).
  \end{equation}
Recall that $L$ is bounded on $L_p(0, \tau; H)$ by Lemma \ref{lemLp}. We shall now prove that $Q$ is bounded on $L_p(0, \tau; H)$. 
Let $g \in C_c^\infty(0, \tau; H)$. By Proposition~\ref{prop:sg-extrapolation} we have
\begin{align*}
 \norm{ (Qg) (t)}_H  
\le  &\; \int_0^t \tfrac{C}{t-s} \norm{ e^{-(t{-}s)  A(t) / 2} (\calA(t){-}\calA(s)) \calA(s)^{-1} g(s)}_{H}   \,\ds \\
\le  &\; \int_0^t \tfrac{C'}{(t{-}s)^{ \nicefrac{3}{2} } } \norm{  (\calA(t){-}\calA(s)) \calA(s)^{-1} g(s)}_{V'}   \,\ds.
\end{align*}
Since $\norm{ \calA(t) x}_{V'} = \sup_{\norm{v}_V=1}
\left|\form{t}[x][v]\right|$, we use the regularity assumption
(\ref{eq:dini-regularity-of-forms}) to bound $Qg$ further by
\begin{equation}\label{truc}
   \norm{ (Qg) (t)}_H    \le  \int_0^t \tfrac{C'}{(t{-}s)^{\nicefrac{3}{2}}} \, \omega(t{-}s) \, \norm{ \calA(s)^{-1} g(s)}_{V}   \,\ds.  
\end{equation}
Now we estimate $\norm{ \calA(s)^{-1} g(s)}_{V}$. By coercivity
\begin{align*}
\al \norm{ \calA(s)^{-1} g(s)}_{V}^2 \le&\; \Re \form{s}[\calA(s)^{-1}g(s)][\calA(s)^{-1}g(s)]\\
=&\; \Re \langle \calA(s) \calA(s)^{-1}g(s), \calA(s)^{-1}g(s) \rangle\\
=&\; \Re \sprod{g(s)}{\calA(s)^{-1}g(s)}\\
\le&\; \norm{g(s)}_H^2 \norm{\calA(s)^{-1}}_{\BOUNDED(H)}.
\end{align*}
We obtain from (\ref{truc}) that 
\begin{equation}\label{est00}
\norm{ (Qg) (t)}_H    \le  \int_0^t \tfrac{C'}{(t{-}s)^{\nicefrac{3}{2}}} \, \omega(t{-}s) \,   \norm{\calA(s)^{-1}}_{\BOUNDED(H)}^{\einhalb}  \norm{g(s)}_H \,\ds. 
\end{equation}
Now, once we replace $A(t)$ by $A(t){+}\mu$, (\ref{truc}) is valid
with a constant independent of $\mu \ge 0$ by
Proposition~\ref{prop:sg-extrapolation}\ref{item:tout}. Using the
estimate
\[
 \norm{(\calA(s) + \mu)^{-1}}_{\BOUNDED(H)} \le \tfrac{1}{\mu},
 \]
 in \eqref{est00} for  $A(s){+}\mu$ we see that
\[
   \norm{ (Qg) (t)}_H \le    \tfrac{C'}{\sqrt{\mu}}  \int_0^t \tfrac{\omega(t{-}s)}{(t{-}s)^{\nicefrac{3}{2}}}  \,   \norm{g(s)}_H \,\ds.
\]
It remains to see that the operator $S$ defined  by 
 \[
   Sh(t) := \int_0^t \tfrac{\omega(t{-}s)}{(t{-}s)^{\nicefrac{3}{2}}}  h(s) \, \ds
\]
is bounded on $L_p(0,\tau; \RR)$. Observe that $S$ is an integral
operator with kernel function $(t,s) \mapsto \eins_{[0,
  t]}(s)\tfrac{\omega(t{-}s)}{(t{-}s)^{\nicefrac{3}{2}}}$. Hence by
assumption (\ref{eq:Dini-3-2-condition}) it is integrable with respect
to each of the two variables with uniform bound with respect to the
other variable.  This implies that $S$ is bounded on $L_1(0, \tau; H)$
and on $L_\infty(0, \tau; H)$ and hence bounded on $L_p(0, \tau;H)$.

\noindent It follows that $Q$ is bounded on $L_p(0, \tau; H)$ with
norm of at most $\tfrac{C''}{\sqrt{\mu}}$ for some constant
$C''$. Taking then $\mu $ large enough makes $Q$ strictly contractive
such that $(I-Q)^{-1}$ is bounded by the Neumann series. Then, for $f
\in C_c^\infty(0, \tau; H)$, (\ref{eq:intermediate}) can be rewritten
as
\[
    A(\cdot)u(\cdot) =  (I-Q)^{-1} L f(\cdot).
\]
For  general $u_0 \in (H, \DOMAIN(A(0)))_{1-\nicefrac1{p}, p}$ we suppose in
addition to (\ref{eq:Dini-3-2-condition}) that (\ref{eq:p-Dini})
holds. Lemma~\ref{lem:operateur-R} shows that $R u_0 \in L_p(0, \tau;
H)$.  As previously we conclude that 
\[
   A(\cdot)u(\cdot) = (I-Q)^{-1} (Lf + R u_0),
\]
whenever $f \in C_c^\infty(0, \tau; H)$. Thus taking the $L_p$ norms we have
\[ 
   \norm{A(\cdot)u(\cdot)}_{L_p(0, \tau;H)} \le C \norm{(Lf + R u_0)}_{L_p(0, \tau;H)}.
\]
We use again the previous estimates on $L$ and $R$ to obtain
\[
    \norm{ A(\cdot) u(\cdot)}_{L_p(0, \tau; H)} \le C' \left[ \norm{f}_{L_p(0, \tau; H)} + \norm{u_0}_{(H, \DOMAIN(A(0)))_{1-\nicefrac1{p}, p}} \right].
\]
Using the equation (\ref{eq:evol-eq}) we obtain a similar  estimate  for $u'$ and so
\[
    \norm{ u'(\cdot)}_{L_p(0, \tau; H)} +  \norm{ A(\cdot) u(\cdot)}_{L_p(0, \tau; H)} \le C'' \left[ \norm{f}_{L_p(0, \tau; H)} + \norm{u_0}_{(H, \DOMAIN(A(0)))_{1-\nicefrac1{p}, p}} \right].
\]
 We write $u(t) = A(t)^{-1} A(t) u(t)$ and use one again the fact that  the norms of $A(t)^{-1}$ on $H$ are uniformly bounded we obtain
\[
   \norm{ u(t) }_{L_p(0, \tau; H)} \le C_1 \norm{ A(\cdot) u(\cdot)}_{L_p(0, \tau; H)} \le C_2 \left[ \norm{f}_{L_p(0, \tau; H)} + \norm{u_0}_{(H, \DOMAIN(A(0)))_{1-\nicefrac1{p}, p}} \right].
\]
 We conclude therefore that the following a priori estimate holds
\begin{align}  \label{eq:etoile}
  & \norm{ u }_{L_p(0, \tau; H)} +  \norm{ u' }_{L_p(0, \tau; H)} +  \norm{ A(\cdot)u(\cdot)  }_{L_p(0, \tau; H)} \nonumber\\
    &\; \le \; C \left[ \norm{f}_{L_p(0, \tau; H)} + \norm{u_0}_{(H, \DOMAIN(A(0)))_{1-\nicefrac1{p}, p}} \right],
\end{align}
where the constant $C$ does not depend on $f \in C_c^\infty(0, \tau; H)$.

 Now let
$f \in L_p(0, \tau; H)$ and $(f_n) \subset C_c^\infty(0, \tau; H)$ be an
approximating sequence that converges in $L_p$ and pointwise almost
everywhere. For each $n$,  denote by $u_n$ the solution of (\ref{eq:evol-eq}) with right hand side $f_n$. 
We apply (\ref{eq:etoile}) to $u_n - u_m$ and we see that 
there exists $u \in W^1_p(0, \tau;
H)$ and $v \in L_p(0, \tau; H)$ such that 
\begin{equation}  \label{eq:les-3}
u_n \overset{L_p}{\relbar\joinrel\relbar\joinrel\rightarrow} u
\qquad
u_n' \overset{L_p}{\relbar\joinrel\relbar\joinrel\rightarrow} u'
\quad\text{and}\quad
A(\cdot)u_n(\cdot) \overset{L_p}{\relbar\joinrel\relbar\joinrel\rightarrow} v
\end{equation}
By extracting a subsequence, we may assume that these limits hold in the pointwise a.e.  sense as well. For a fixed $t$, the operator $A(t)$ is closed and so $v (\cdot) = A(\cdot) u(\cdot)$. 
Passing to the limit in the equation
$$u_n'(t) + A(t)\, u_n(t) = f_n(t)$$
shows that 
$$u'(t) + A(t) u(t) = f(t)$$
for a.e. $t \in (0, \tau)$. On the other hand, by  Sobolev embedding, $(u_n)$ is bounded in $C([0,\tau];H)$ and hence 
$u(0) = u_0$  since $u_n(0) = u_0$ by the definition of $u_n$. 
We conclude that $u$ satisfies 
\[
    u'(t) + A(t)\, u(t) = f(t)\qquad u(0)=u_0
\]
 in the $L_p$ sense. This means
that $u$ is a solution to (\ref{eq:evol-eq}). 
Moreover, (\ref{eq:etoile}) transfers from $u_n$ to $u$. The uniqueness of the solution $u$ follows from the a priori estimate
(\ref{eq:etoile}) as well.
\end{proof}
\begin{proof}[Proof of Corollary~\ref{cor1}]
The result follows from Theorem~\ref{thm:main} and  the observation that 
\[
(H,\DOMAIN(A(0)))_{\nicefrac1{2},2} = [(H,\DOMAIN(A(0))]_\einhalb =  \DOMAIN( (\delta + A(0))^\einhalb),
\]
see e.g. \cite[Corollaries 4.37 and 4.30]{Lunardi:book}.
\end{proof}
\begin{proof}[Proof of Corollary~\ref{coro:main}]
  By the definition of maximal regularity, one may modify the
  operators $A(t), \ 0 \le t \le \tau$, on a set of Lebesgue measure
  zero. Therefore, we may assume without loss of generality that the
  mapping $t \mapsto \form{t}[u][v]$ is right continuous.  We may assume again that the operators $A(\cdot)$ are all invertible.
  We apply Corollary~\ref{cor1} to the evolution equation
\[
\left\{ \begin{aligned} 
u_j'(t) + A(t) u_j(t) &= f(t)\qquad  t \in (t_{j}, t_{j+1})\\
 u_j(t_{j}) &= u_{j-1}(t_{j}),
\end{aligned}
\right.
 \]
 since it is obvious that the assumed $\al$-H\"older continuity for
 some $\al > \einhalb$ implies both (\ref{eq:Dini-3-2-condition}) and
 (\ref{eq:2-Dini}).  The solution $u_{j}$ is in $ W^1_2(t_{j}, t_{j+1}
 ; H)$ provided the initial data satisfies
\[
   u_{j}(t_{j}) := u_{j-1}(t_{j})  \in \DOMAIN(A(t_{j})^\einhalb).
\]
Note that the endpoint $u_{j-1}(t_{j})$ is well defined since $u_j \in
C([t_{j}, t_{j+1}];H)$ by \cite[XVIII Chapter 3, p.\ 513]{DL92}.  In
order to obtain a solution $u \in W^1_2(0, \tau; H)$, we glue the
solutions $u_j$. That is, we set $u(t) = u_j(t)$ for $t \in [t_j,
t_{j+1}]$. What remains then to prove is that $u(t_j) \in
\DOMAIN(A(t_{j})^\einhalb)$, where $u \in W^1_2(0, \tau ; V')$ is the
solution in $V'$ given by Lions' theorem.

Fix one of the discontinuity points $t_j$ and consider the autonomous equation
\[
   v'(s) + A(t_j^-) v(s) = f(s), \qquad v(0)=0.
\]
By maximal regularity of $A(t_j^-)$, its solution $v(s) = \int_0^s
e^{-(s{-}r) A(t_j^-)} \, f(r)\,\dr$ satisfies $v(s) \in
\DOMAIN(A(t_j^-))$ for almost all $s$.  Choose a sequence $(s_n)$
converging to $t_j$ from the left such that $v(s_n) \in
\DOMAIN(A(t_j^-))$.  Since $A(t_j^-)$ is an accretive and sectorial
operator it has a bounded $H^\infty$-calculus of some angle $<
\pihalbe$. Hence $A(t_j^-)$ and its adjoint admit square-function
estimates of the form:
\[
   \int_0^\infty \norm{ A(t_j^-)^\einhalb e^{-r\,A(t_j^-)} x}_H^2\,\dr \le C \norm{x}_H^2 \  \text{for all } x \in H,
\]
see e.g. \cite[Section~8]{McIntosh:H-infty-calc}. It follows that 
\begin{equation}
  \label{eq:square-function}
\begin{split}
 & \; \int_0^{s_n} \bignorm{ A(t_j^-)^\einhalb  e^{-(t_j{-}r) A(t_j^-)} f(r)}_H \,\dr \\
= & \; \int_0^{s_n} \sup_{\norm{h}_H\le 1} \sprod{f(r)}{{A(t_j^-)^*}^\einhalb  e^{-(t_j{-}r) A(t_j^-)^*} h} \,\dr \\
\le & \; \bignorm{f}_{L_2(0,\tau;H)} 
\sup_{\norm{h}_H\le 1}  \Bigl(\int_0^{s_n} \bignorm{ {A(t_j^-)^*}^\einhalb  e^{-(t_j{-}r) A(t_j^-)^*} h}_H^2 \,\dr\Bigr)^\einhalb\\
 \;  \le& \; C \bignorm{f}_{L_2(0,\tau;H)}.
\end{split}
\end{equation}
Thus,  (\ref{eq:square-function}) implies that the
sequence $(v(s_n))_{n\ge 0}$ is bounded in the Hilbert space $\DOMAIN(
A(t_j^-)^\einhalb )$. It has a weakly convergent subsequence. By extracting a subsequence, we may assume that $(v(s_n))_{n\ge 0}$ converges
weakly to some $v$  in $\DOMAIN( A(t_j^-)^\einhalb)$.  But the continuity
of the solution $v(\cdot)$ implies that $v(s_n)$ tends also to
$v(t_j)$ in $H$. Therefore,
\[
  v(t_j) =  \int_0^{t_j} e^{-(t_j{-}r) A(t_j^-)} \, f(r)\,\dr   \in \DOMAIN(A(t_j^-)^\einhalb).
\]
In particular, 
\begin{equation}\label{formule-voulue}
   \int_{t_{j{-}1}}^{t_j} e^{-(t_j{-}r) A(t_j^-)} \, f(r)\,\dr \in \DOMAIN(A(t_j^-)^\einhalb ).
\end{equation}
On the other hand, as  in the proof of Lemma~\ref{lem:AT-formula-for-u},
we have for  all $t> t_{j{-}1}$
\begin{equation}  \label{eq:formule-2}
  \begin{split}
  &   u(t) - e^{-(t{-}t_{j-1})  \calA(t) } u(t_{j-1})   \\
= & \; \int_{t_{j{-}1}}^t e^{-(t{-}s)  \calA(t)} (\calA(t) - \calA(s)) u(s)\,\ds +  \int_{t_{j{-}1}}^t e^{-(t{-}s) \calA(t)} f(s)\,\ds.
  \end{split}
\end{equation}
By analyticity of the semigroup $e^{-s \, \calA(t_j^-) }$ it follows
that $e^{-(t_j{-}t_{j-1})  \,\calA(t_j^-) } u(t_{j-1}) \in
\DOMAIN(A(t_j^-)^\einhalb)$.  Now we prove that $\int_{t_{j{-}1}}^{t_j}
e^{-(t_j {-}s)  \,\calA(t_j^-)} (\calA(t_j^-) - \calA(s)) u(s)\,\ds \in
\DOMAIN(A(t_j^-)^\einhalb)$.  
It is enough to prove that 
\begin{equation}\label{eqHH}
 \int_{t_{j{-}1}}^{t_j} A(t_j^-)^\einhalb e^{-(t_j{-}s) \calA(t_j^-)} (\calA(t_j^-) - \calA(s)) u(s)\,\ds \in H.
\end{equation}
To this end, let $h \in H$ be such that $\norm{h}_H \le 1$. By Proposition \ref{prop:sg-extrapolation}, (c),  we have 
\begin{equation}\label{eqVA}
 \norm{{A(t_j^-)^*}^\einhalb  e^{-(t_j {-} s) A(t_j^-)^*} h}_V \le C | t_j - s|^{-1}
 \end{equation}
Thus, since the form is $C^\alpha$ on $(t_{j-1}, t_j)$, we have for every  small $\epsilon > 0$
\begin{align*}
    & \Bigl| \langle \int_{t_{j-1}}^{t_j -\epsilon}  A(t_j^-)^\einhalb e^{-(t_j{-}s) \calA(t_j^-)} (\calA(t_j^-) - \calA(s)) u(s)\, \ds, h \rangle \Bigr|\\
=   & \; \Bigl| \int_{t_{j-1}}^{t_j -\epsilon}\!\! \form{t_j}[u(s)][{A(t_j^-)^*}^\einhalb  e^{-(t_j {-} s) A(t_j^-)^*} h] - \form{s}[u(s)][{A(t_j^-)^*}^\einhalb  e^{-(t_j {-} s) A(t_j^-)^*} h]  \,\ds \Bigr|\\
\le & \; C \int_{t_{j-1}}^{t_j -\epsilon} | t_j - s|^{\alpha} \norm{u(s)}_V \norm{{A(t_j^-)^*}^\einhalb  e^{-(t_j {-} s) A(t_j^-)^*} h}_V  \,\ds\\
\le & \; C' \int_{t_{j-1}}^{t_j} | t_j - s|^{\alpha-1} \norm{u(s)}_V\,\ds.
\end{align*}
Taking the supremum over all $h \in H$ of norm we obtain
\[  \norm{ \int_{t_{j-1}}^{t_j -\epsilon}  A(t_j^-)^\einhalb e^{-(t_j{-}s) \calA(t_j^-)} (\calA(t_j^-) - \calA(s)) u(s)\, \ds}_H \le 
C'' \norm{u}_{L_2(0, \tau;V)}.
\]
Since this is true for all  $\epsilon > 0$ we obtain
\eqref{eqHH}. We conclude from this, \eqref{formule-voulue} and
\eqref{eq:formule-2} that $u_{j-1}(t_j) = u(t_j)  \in \DOMAIN( A(t_j^-)^\einhalb)$. Finally, the latter space coincides 
with $\DOMAIN( A(t_j^-)^\einhalb) = \DOMAIN( A(t_j^+)^\einhalb)$ by the assumptions of the corollary.  
\end{proof}

\section{Operator-valued pseudo-differential operators}\label{sec-pseudo}
Given a Hilbert space $H$, our aim in this section is to prove results
on boundedness on $L_2(\RR^n ; H)$ for pseudo-differential operators
with minimal smoothness assumption on the symbol. The main results we
will show here were proved in \cite{MuNa81} in the scalar case
(i.e. $H {=} \CC$), see also \cite{ANV03}. The operator-valued version
follows the lines in \cite{MuNa81} and we  give the details here
for the sake of completeness. Let us mention the paper
\cite{HytonenPortal:pseudo} where results on $L_p$--boundedness of
pseudo-differential operators with operator-valued symbols are proved
even when $H$ is not a Hilbert space. We do not appeal to the results
from \cite{HytonenPortal:pseudo} in order to avoid assuming continuity
and concavity assumptions on the function $\om$ in Theorem
\ref{thm:main}.

Let $H$ be a Hilbert space on $\CC$, with scalar product
$\sprod{\cdot}{\cdot}$ and associated norm $\norm{\cdot}_H$.
 \[
   \sigma: \RR^n \times \RR^n \to \BOUNDED(H)
\]
be bounded measurable. We define for $f $ in the Schwartz space $\SCHWARTZ(\RR^n; H)$
\[
    T_\sigma f(x) := \tfrac1{(2\pi)^n} \int_{\RR^n} \sigma(x, \xi) \widehat{f}(\xi) e^{i x \xi}\ \dxi.
\]
where we write $\widehat f$ for the Fourier transform of $f$.  We
shall also use the notation $|\xi|$ for the Euclidean distance in
$\RR^n$ and write henceforth $\langle\xi\rangle := \sqrt{1 +
  |\xi|^2}$.  For the rest of this section, we will ignore the
normalisation constant in the definition of the Fourier transform.

\begin{theorem}\label{thmpseudo}
Suppose that there exists a non-decreasing function  $\omega: [0, \infty) \to [0, \infty)$ such that 
\[ 
\norm{\partial_\xi^\al \sigma(x, \xi)}_{\BOUNDED(H)} \le C_\al \langle\xi\rangle^{-|\al |}
\]
and
\[
\norm{\partial_\xi^\al \sigma(x, \xi) - \partial_\xi^\al \sigma(x', \xi)}_{\BOUNDED(H)} \le C_\al \langle\xi \rangle^{-| \al | } \omega(| x-x'|)
\]
for all $| \al | \le [\nicefrac{n}{2}] + 2$ and some positive  constant $C_\al$.  Suppose in addition that 
\[
\int_0^1 \omega(t)^2 \, \tfrac{\dt}t < \infty,
\]
then $T_\sigma$ is a bounded operator on $ L_2(\RR^n ; H)$.\footnote{In
  \cite{HM99}, $L_2(\RR; H)$--boundedness of $T_\sigma$ is claimed for
  symbols $\sigma: \RR {\times} \RR \to \BOUNDED(H)$
  that admit a bounded holomorphic extension to a double sector of
  $\CC$ in the variable $\xi$, without  any kind of
  regularity in the variable $x$. }
\end{theorem}

\begin{proof} Let $\varphi \in C^\infty(\RR^n)$ be a non-negative
  function with support in the unit ball such that $\int_{\RR^n}
  \varphi(x) \,\dx = 1$.  Fix a constant $\delta \in (0,1)$ and define
  the symbols
\[
   \sigma_1(x, \xi) := \int_{\RR^n} \varphi(y) \, \sigma(x - \frac{y}{\langle\xi \rangle^{\delta}},  \xi) \, \dy
\]
and
\[
   \sigma_2(x, \xi) := \sigma(x, \xi) - \sigma_1(x, \xi). 
\]
It is clear that 
\[
    \sigma_1(x, \xi) = \int_{\RR^n} \varphi(\langle\xi\rangle^\delta (x-y)) \, \sigma(y,\xi)\langle\xi\rangle^{n \delta} \,\dy
\]
and one checks that 
\begin{equation}\label{MN00}
 \norm{ \partial_x^\beta \partial_{\xi}^\al \sigma_1(x, \xi)  }_{\BOUNDED(H)} 
 \le c_{\al, \beta}  \langle \xi\rangle^{- |\al | + \delta | \beta |} \le c_{\al, \beta} \langle \xi\rangle^{-\delta(|\al | - | \beta |)}
 \end{equation}
 and
 \begin{equation}\label{MN01}
 \norm{\partial_\xi^\al \sigma_2(x,\xi)  }_{\BOUNDED(H)} \le c_\al \omega(\langle \xi\rangle^{-\delta}) \langle \xi\rangle^{-|\al|}
 \end{equation}
 for $|\al | \le [\frac{n}{2}] + 2$ and all $\beta$.  Using
 \eqref{MN01} we conclude by the next theorem that $T_{\sigma_2}$ is
 bounded on $L_2(\RR^n; H)$. The boundedness of $T_{\sigma_1}$ on
 $L_2(\RR^n; H)$ follows from \eqref{MN00} and Theorem 1 in
 \cite{Mu73}. Note that it is assumed there that the symbol is
 $C^\infty$ but the estimate needed in Theorem 1 is exactly
 \eqref{MN00} with $|\al | \le [\frac{n}{2}] + 2$.
\end{proof} 
\begin{theorem}\label{thm:MuramatoNagase-2.2}
  Let $\delta \in (0, 1)$ and $w: [0, 1]\to\RR_+$ be a non-decreasing
  measurable function satisfying
\[
   \int_0^1   \omega(t)^2\,\tfrac{\dt}t < \infty.
\]
If a bounded strongly measurable symbol $\sigma: \RR^n \times \RR^n
\to \BOUNDED(H)$ satisfies
\begin{equation}  \label{eq:cond2.8}
  \norm{  \partial_\xi^\al \sigma(x, \xi) }_{\BOUNDED(H)}
\le
C_\al \bracket{\xi}^{-\alpha} \, \omega( \bracket{\xi}^{-\delta}) 
\end{equation}
for $|\al| \le \kappa := [\nicefrac{n}2]+1$, then $T_\sigma$ is bounded on $L_2(\RR^n; H
)$.
\end{theorem}

\begin{proof}
  Let $\varphi$ be a non-negative $C_c^\infty$ function satisfying 
  $\varphi(\xi) = 1$ for $|\xi| \le  2$ and $\varphi(\xi) = 0$ for $|\xi|>3$.
  Then we may rewrite 
\[
   \sigma(x, \xi) = \varphi(\xi) \sigma(x, \xi) 
             \; + \; (1{-}\varphi(\xi)) \sigma(x, \xi)
                  = \sigma_1(x, \xi) + \sigma_2(x, \xi)
\]
and treat both parts separately. For the first part, let $f \in \SCHWARTZ(\RR^n ; H)$ and $h \in H$. Then
 \begin{align*}
 \sprod{ (T_{\sigma_1} f)(x) }{h}
 =  & \; \int_{\RR^n} e^{ix\cdot \xi} \sprod{ \sigma_1(x, \xi)  \widehat f(\xi)}{h}\,\dxi \\
 =  & \; \int_{\RR^n} \sprod{ f(y) }{\int_{\RR^n} e^{i(x-y)\cdot \xi}  \sigma_1(x, \xi)^*h \,\dxi}\,\dy\\
 =: & \; \int_{\RR^n} \sprod{ f(y) }{ K(x, x{-}y) }\,\dy.
\end{align*}
By Plancherel's theorem,
\begin{align*}
\int_{\RR^n} \norm{\bracket{z}^{2\kappa} K(x, z) }^2\,\dz 
\le & \; 
\sum_{|\al|\le 2\kappa} c_\al \int_{\RR^n} \norm{ z^\al K(x, z) }^2\,\dz \\
= & \; 
\sum_{|\al|\le 2\kappa} c_\al \int_{\RR^n} \norm{ \partial_\xi^\al \sigma_1(x, \xi)^* h }^2\,\dz 
=: C_1 \norm{h}^2, 
\end{align*}
where $C_1$ is finite due to the  support of $\sigma_1$.
By the Cauchy-Schwarz inequality, 
\begin{align*}
   &   \; \norm{ T_{\sigma_1} f }_{L_2(\RR^n; H) }^2 \\
  = &  \; \int_{\RR^n} \sup_{h \in H, \norm{h}\le 1} \Bigl|\int_{\RR^n} \sprod{f(y)}{K(x, x{-}y)}\,\dy  \Bigr|^2 \,\dx\\
\le &  \; \int_{\RR^n} \sup_{h \in H, \norm{h}\le 1} \Bigl( \int_{\RR^n} \bracket{x{-}y}^{-2\kappa} \norm{f(y)}_H^2 \, \dy\Bigr) 
                     \Bigl( \int_{\RR^n}  \norm{\bracket{x{-}y}^{2\kappa} K(x, x{-}y) }_H^2  \,\dy \Bigr)\,\dx\\
\le & \; C_1 \int_{\RR^n} \int_{\RR^n}  \bracket{x{-}y}^{-2\kappa} \,\dx \;\norm{ f(y) }_H^2 \,\dy = C_1 C_2 \norm{f}_{L_2(\RR^n; H)}^2.
\end{align*}
Thus, $T_{\sigma_1}$ is bounded on $L_2(\RR^n; H)$. 

\medskip

\noindent Next we  show boundedness of $T_{\sigma_2}$. Recall that $\supp(\sigma_2) \subset \{ (x, \xi) \SUCHTHAT |\xi| \ge 2 \}$. Let
$\phi \in C_c^\infty$ such that $\supp(\phi) \subseteq [1,2]$ and $\int_0^\infty |\phi(t)|^2\,\tfrac{\dt}t =
1$.   Let $f \in \SCHWARTZ(\RR^n; H)$ and  $h \in H$. 
\begin{align*}
 \sprod{ (T_{\sigma_2} f)(x)}{h}
= & \; \int_{\RR^n} e^{ix\cdot \xi} \sprod{\sigma_2(x, \xi) \widehat f(\xi)}{h}\,\dxi \\
= & \; \int_0^\infty \int_{\RR^n} e^{ix\cdot \xi} \phi(|t \xi|)^2  \sprod{\widehat f(\xi)}{\sigma_2(x, \xi)^* h} \,\dxi \,\tfrac{\dt}t\\
= & \;  \int_0^{1} \int_{\RR^n}   \sprod{e^{ix\cdot \xi} \phi(|t \xi|) \widehat f(\xi)}{\phi(|t \xi|)\, \sigma_2(x, \xi)^* h}  \,\dxi\, \tfrac{\dt}t \\
= & \;  \int_0^{1} \int_{\RR^n}   \sprod{t^{-n} e^{i \nicefrac{x}{t} \cdot \xi} \phi(|\xi|) \widehat f(\nicefrac{\xi}t)}{\phi(|\xi|)\, \sigma_2(x, \nicefrac{\xi}t)^* h}  \,\dxi\, \tfrac{\dt}t.
\end{align*}
Recall that $\phi \in \SCHWARTZ(\RR)$, so that by Plancherel's theorem
\[
 \sprod{ (T_{\sigma_2} f)(x)}{h}
= 
 \int_0^{1} \int_{\RR^n}   \sprod{ K_1(t, x, z)}{ K_2(t, x, z)}  \,\dz\, \tfrac{\dt}t, 
\]
where $K_1$ and $K_2$ are the respective Fourier transforms
\[
   K_1(t, x, z) = \int_{\RR^n} t^{-n} e^{i( \nicefrac{x}{t}-z)\cdot \xi}  \phi(|\xi|) \widehat f(\nicefrac{\xi}t)\,\dxi 
= \int_{\RR^n} e^{i(x-tz)\cdot \xi}  \phi(|t\xi|) \widehat f(\xi)\,\dxi
\]
and
\[
   K_2(t, x, z) = \int_{\RR^n} e^{-iz \cdot \xi} \phi(|\xi|)\, \sigma_2(x, \nicefrac{\xi}t)^* h \,\dxi.
\]

\noindent Now, by the Cauchy-Schwarz inequality,
\begin{align*}
    & \; \bigl|  \sprod{ (T_{\sigma_2} f)(x)}{h} \bigr|^2 \\
 =  &\;  
         \Bigl| \int_0^{1} \int_{\RR^n}   \sprod{\bracket{z}^{-\kappa}K_1(t, x,z)}{\bracket{z}^{\kappa} K_2(t, x, z) } \,\dz \,\tfrac{\dt}t \Bigr|^2 \\
\le & \; 
         \Bigl(  \int_0^{1} \int_{\RR^n} \bignorm{\bracket{z}^{-\kappa} K_1(t, x,z)}_H^2  \,\dz \,\tfrac{\dt}t\Bigr) 
         \Bigl(  \int_0^{1} \int_{\RR^n} \bignorm{\bracket{z}^{\kappa} K_2(t, x, z)}_H^2  \,\dz \,\tfrac{\dt}t\Bigr).
\end{align*}
Observe that $\bracket{z}^{\kappa} K_2(t, x, z) = \FOURIER\bigl(
(I{-}\Delta_\xi)^{\nicefrac{\kappa}{2}} \phi(|\xi|) \sigma_2(x, \nicefrac{\xi}t)^* h
\bigr)(z)$.  Recall that $\phi \in C_c^\infty(\RR)$ has its support in
$[1, 2]$, so that, for $|\xi| \ge 1$, and using the growth assumption
(\ref{eq:cond2.8}) on derivatives of $\sigma$,
\begin{align*}
& \bigl| \partial_\xi^\beta  \phi(|\xi|) |  \le C \, \bracket{\xi}^{-|\beta|} \le C \\
& \bignorm{ \partial_\xi^\gamma  \sigma_2(x, \nicefrac{\xi}t)^* h } 
 \le C' \, t^{-|\gamma|} \, \norm{h} \, \bigl( |\nicefrac{\xi}t|^{-|\gamma|} \omega(|\nicefrac{\xi}t|^{-\delta}) \bigr) 
 \le C''\norm{h}  \, \omega(t^{+\delta}).
\end{align*}
Note that  we used here  the monotonicity  of $\omega$. The Dini type condition on $\omega$ then gives
\[
 \int_0^{1} \int_{\RR^n} \bignorm{\bracket{z}^{\kappa} K_2(t, x, z)}_H^2  \,\dz \,\tfrac{\dt}t
\le C'' \norm{h}_H^2 \int_0^{1} \int_{1\le |\xi|^2 \le 2} \omega(t^{+\delta})^2\,\dxi  \,\tfrac{\dt}t =: C_2 \norm{h}_H^2.
\]
We conclude by observing that, again by Plancherel,
\begin{align*}
    & \; \int_{\RR^n} \sup_{h\in H, \norm{h}\le 1} \bigl| \sprod{  (T_{\sigma_2} f)(x)}{h} \bigr|^2\,\dx \\
\le & \; C_2 \int_{\RR^n}  \int_0^{1} \int_{\RR^n} \bracket{z}^{-2\kappa}  \bignorm{ K_1(t, x,z)}_H^2 \,\dz \,\tfrac{\dt}t \,\dx\\
=   & \; C_2  \int_0^{1}  \int_{\RR^n} \bracket{z}^{-2\kappa}  \int_{\RR^n}  \bignorm{ K_1(t, x,z)}_H^2 \,\dx  \,\dz \,\tfrac{\dt}t\\
=   & \; C_2'   \int_0^{1}  \int_{\RR^n} \bracket{z}^{-2\kappa}  \int_{\RR^n} |\phi(|t\xi|)|^2 \norm{ \widehat f(\xi) }_H^2 \,\dxi  \,\dz \,\tfrac{\dt}t\\
\le & \; c_n C_2'   \int_{\RR^n}  \norm{ \widehat f(\xi) }_H^2  \int_0^{\infty} |\phi(|t\xi|)|^2 \,\tfrac{\dt}t \,\dxi\\
= & \, c_n C_2'   \int_{\RR^n}  \norm{ \widehat f(\xi) }_H^2\,  \dxi.\\
\end{align*}

\vspace*{-0.7cm}

\noindent  Therefore, $T_{\sigma_2}$ and hence $T_\sigma$ are bounded on
$L_2(\RR^n; H)$.
\end{proof}

\section{Examples}
In this section we discuss some examples and applications of our
results. We shall focus on two simple but relevant linear problems
which involve elliptic operators. Note that following
\cite{ArendtDierLaasriOuhabaz}, we may also consider quasi-linear
evolution equations.  Our maximal regularity results can be used
to improve some results in \cite{ArendtDierLaasriOuhabaz} on existence
of solutions to quasi-linear problems in the sense that we assume less
regularity with respect to $t$ of the coefficients of the equations. We shall
not pursue this direction here.

\subsection{Elliptic operators}
Define on $H = L_2(\RR^d, \dx)$ the sesquilinear forms
\[
   \form{t}[u][v] = \sum_{k,j=1}^d \int_{\RR^d}a_{kj}(t,x)  \partial_k u \overline{\partial_j v} \ \dx \, \text{for}\  u, v \in  W^{1}_2(\Omega).
\]
We assume that $a_{kj}: [0, \tau] \times \RR^d \to \CC$ such that:
\[
    a_{kj} \in L_\infty([0,\tau]\times \RR^d) \, \text{for} \ 1 \le k, j \le d,
\]
and
\[
   \Re \sum_{k,j=1}^d a_{kj}(t,x) \xi_k \overline{\xi_j}  \ge \nu |\xi |^2 \, \text{for all }\  \xi \in \CC^d  \ \text{and a.e. } (t,x) \in [0, \tau]\times \RR^d.
\]
Here $\nu > 0$ is a constant independent of $t$. \\
It is easy to check that $\form{t}$ is $ W^{1}_2(\RR^d)$-bounded and
quasi-coercive. The associated operator with $\form{t}$ is the
elliptic operator given by the formal expression
\[
   A(t) u = - \sum_{k,j=1}^d \partial_j \left( a_{kj}(t,.) \partial_k u \right).
\]
In addition to the above assumptions we assume that for some constant
$M$ and $\alpha > \einhalb$
\begin{equation}\label{aHo}
\lvert a_{kj}(t,x) - a_{kj}(s, x) \rvert \le M \lvert t-s\rvert^\alpha \, \text{for a.e.} \ x \in \RR^d \ \text{and all}\ t, s \in [0, \tau].
\end{equation}
By the Kato square root property, it is known that $\DOMAIN(A(0)^\einhalb) =
W^{1}_2(\RR^d)$, see
\cite{AuscherHofmannLaceyMcIntoshTchamitchian}. Therefore, applying
Corollary~\ref{coro:main} we conclude that for every $f \in
L_2(0,\tau; H)$ the problem
\[
\left\{
  \begin{array}{rcl}
     u'(t) - \sum_{k,j=1}^d \partial_j \left( a_{kj}(t,.) \partial_k u(t) \right) &=& f(t), \ t \in (0, \tau] \\
     u(0)&=&u_0 \in  W^{1}_2(\RR^d) 
  \end{array}
\right.
\]
has a unique solution $u \in W^{1}_2(0,\tau; H) \cap L_2(0,\tau;
W^{1}_2(\RR^d))$.

\subsection{Time-dependent Robin boundary conditions}\label{subsection:example1}
We consider here the Laplacian on a domain $\Omega$ with a time
dependent Robin boundary condition
\begin{equation}\label{rob}
\partial_\nu u(t) + \beta(t,.) u = 0 \text{ on } \Gamma= \partial\Omega,
\end{equation} 
for some function $\beta: [0,\tau] \times \Gamma \to \RR$. This
example is taken from \cite{ArendtDierLaasriOuhabaz}. The difference
here is that we assume less regularity on $\beta$ and also that we can
treat maximal $L_p$--regularity for $p \in (1, \infty) $ whereas the
results in \cite{ArendtDierLaasriOuhabaz} are restricted to the case
$p{=}2$.

Let $\Omega$ be a bounded domain of $\RR^d$ with Lipschitz boundary
$\Gamma = \partial\Omega$ and denote by $\sigma$ the
$(d{-}1)$-dimensional Hausdorff measure on $\Gamma$.  Let $\beta:
[0,\tau] \times \Gamma \to \RR$ be a bounded measurable function which
is H\"older continuous w.r.t.\ the first variable, i.e.,
\begin{equation}\label{lipbeta}
 	\lvert \beta(t,x) - \beta(s, x) \rvert \le M \lvert t-s\rvert^\alpha
\end{equation}
for some constants $M$, $\alpha > \einhalb$ and all $t, s \in[0,
\tau], \ x \in \Gamma$. We consider the symmetric form
\begin{equation}\label{formbeta}
	\form{t}[u][v] = \int_\Omega \nabla u \nabla v\ d x + \int_\Gamma \beta(t, .) u v\ d\sigma, \, u, v \in   W^{1}_2(\Omega).
\end{equation}
The form $\form{t}$ is $ W^{1}_2(\Omega)$-bounded and quasi-coercive.
The first statement follows from the continuity of the trace operator
and the boundedness of $\beta$.  The second one is a consequence of
the inequality
\begin{equation}\label{trace-comp}
\int_\Gamma \lvert u \rvert^2 \ d\sigma \le \varepsilon \norm u_{ W^{1}_2}^2 + c_\varepsilon \norm u_{L_2(\Omega)}^2,
\end{equation}
which is valid for all $\varepsilon > 0$ ($c_\varepsilon$ is a constant
depending on $\varepsilon$).  Note that $\eqref{trace-comp}$ is a
consequence of compactness of the trace as an operator from $
W^{1}_2(\Omega)$ into $L_2(\Gamma, d \sigma)$, see \cite[Chap.\ 2 § 6,
Theorem~6.2]{Necas:book67}.

The operator $A(t)$ associated with $\form{t}[\cdot][\cdot]$ on $H:=
L_2(\Omega)$ is (minus) the Laplacian with time dependent Robin
boundary conditions (\ref{rob}).  As in
\cite{ArendtDierLaasriOuhabaz}, we use the following weak definition
of the normal derivative.  Let $v \in W^{1}_2(\Omega)$ such that
$\Delta v \in L_2(\Omega)$.  Let $h \in L_2(\Gamma, d \sigma)$.  Then
$\partial_\nu v = h$ by definition if $\int_\Omega \nabla v \nabla w +
\int_\Omega \Delta v w = \int_\Gamma h w \, d \sigma$ for all $w \in
W^{1}_2(\Omega)$.  Based on this definition, the domain of $A(t)$ is
the set
\[
    \DOMAIN(A(t)) = \{ v \in   W^{1}_2(\Omega): \Delta v \in L_2(\Omega),  \partial_\nu v + \beta(t) v\vert_\Gamma = 0 \},
\]
and for $v\in \DOMAIN(A(t))$ the operator is given by $A(t)v = - \Delta v$.

Observe that the form $\form{t}$ is symmetric, so that $W^{1}_2(\Omega) = \DOMAIN(A(0)^\einhalb)$.
From Corollary~\ref{coro:main} it follows that the heat equation
\begin{equation*}
\left\{  \begin{aligned}
          u'(t)  - \Delta u(t)  = &\;  f(t) \\
                        u(0)    = &\;  u_0 \qquad u_0 \in   W^{1}_2(\Omega)\\
\partial_\nu u(t) + \beta(t,.) u = &\;  0   \qquad \text{ on } \Gamma
          \end{aligned} \right.
\end{equation*}
has a unique solution $u \in W^{1}_2(0,\tau; L_2(\Omega))$ whenever $f \in
L_2(0,\tau ; L_2(\Omega))$.  This example is also valid for more general
elliptic operators than the Laplacian.
    
\medskip Note that in both examples we have assumed $\alpha$-H\"older
continuity in (\ref{aHo}) and (\ref{lipbeta}). We could replace this
assumption by piecewise $\alpha$-H\"older continuity as authorised by
Corollary~\ref{coro:main}.
    
\medskip{\bf Acknowledgements} Both authors wish to thank the
anonymous referee for his or her thorough reading of the manuscript
and his or her suggestions and improvements. We further wish to thank
Dominik Dier for his useful  remarks  on an earlier version of this paper.

\def\SUBMITTED{Submitted}
\def\TOAPPEAR{To appear in }
\def\PREPARATION{In preparation }

\def\cprime{$'$}
\providecommand{\bysame}{\leavevmode\hbox to3em{\hrulefill}\thinspace}

\end{document}